\documentclass[11pt,fleqn]{article}

\usepackage{amsthm,amsmath}
\usepackage{amssymb,dsfont,mathrsfs}

\numberwithin{equation}{section}
\newtheorem{theorem}{Theorem}[section]

\newtheorem{proposition}{Proposition}[section]
\newtheorem{corollary}{Corollary}[section]
\newtheorem{remark}{Remark}[section]

\newtheorem{definition}{Definition}[section]
\def\ba{\boldsymbol{a}}
\def\be{\boldsymbol{e}}

\def\bh{\boldsymbol{h}}
\def\bj{\boldsymbol{j}}

\def\bu{\boldsymbol{u}}
\def\bw{\boldsymbol{w}}
\def\bx{\boldsymbol{x}}
\def\by{\boldsymbol{y}}
\def\bz{\boldsymbol{z}}
\def\bJ{\boldsymbol{J}}
\def\bX{\boldsymbol{X}}
\def\bY{\boldsymbol{Y}}
\def\balpha{\boldsymbol{\alpha}}
\def\bbeta{\boldsymbol{\beta}}
\def\bpi{\boldsymbol{\pi}}

\def\bxi{\boldsymbol{\xi}}
\def\bzero{\mathbf{0}}
\def\bone{\mathbf{1}}
\def\btwo{\boldsymbol{2}}
\def\calA{{\cal A}}
\def\calB{{\cal B}}
\def\calN{{\cal N}}
\def\calS{{\cal S}}
\def\calU{{\cal U}}
\def\calV{{\cal V}}

\def\calL{{\cal L}}
\title{Positive recurrence and transience of  a two-station network with server states}
\author{Toshihisa Ozawa \\ 
Faculty of Business Administration, Komazawa University \\
1-23-1 Komazawa, Setagaya-ku, Tokyo 154-8525, Japan \\
E-mail: toshi@komazawa-u.ac.jp }
\date{\today}

\begin{document}

\maketitle

\begin{abstract} 
We study positive recurrence and transience of a two-station network in which the behavior of the server in each station is governed by a Markov chain with a finite number of server states; this service process can represent various service disciplines such as a non-preemptive priority service and $K$-limited service. 
Assuming that exogenous customers arrive according to independent Markovian arrival processes (MAPs),  we represent the behavior of the whole network as a continuous-time Markov chain and, by the uniformization technique, obtain the corresponding discrete-time Markov chain, which is positive recurrent (transient) if and only if the original continuous-time Markov chain is positive recurrent (resp.\ transient). 
This discrete-time Markov chain is a four-dimensional skip-free Markov modulated reflecting random walk (MMRRW) and, applying several existing results of MMRRWs to the Markov chain, we obtain conditions on which the Markov chain is positive recurrent and on which it is transient. The conditions are represented in terms of the difference of the input rate and output rate of each queue in each induced Markov chain. 
In order to demonstrate how our results work in two-station networks, we give several examples.

\smallskip
{\it Keywards}: Multiclass queueing network, multidimensional reflecting random walk, Markov chain, positive recurrence, stability

\end{abstract}

%%%%%%%%%%%%%%%%%%%%%%%
%
% Section 1
%
%%%%%%%%%%%%%%%%%%%%%%%
%
\section{Introduction} \label{sec:introduction}

Stability and instability of queueing networks has been intensively studied for the last few decades, especially, through a connection between the stability of queueing networks and that of the corresponding fluid models (see, for example, Bramson \cite{Bramson08}, Chen \cite{Chen95}, Dai \cite{Dai95,Dai96}, Down and Meyn \cite{Down97},  Gamarnik and Hasenbein \cite{Gamarnik05} and Meyn \cite{Meyn95}). 
A main result in those studies is that the stability of a fluid model implies stability of corresponding queueing networks (see, for example, Chen \cite{Chen95} and Dai \cite{Dai95}). Several versions of its converse have also been obtained (see, for example, Dai \cite{Dai96} and Gamarnik and Hasenbein \cite{Gamarnik05}). 
It is not doubted that the fluid model is one of the most useful tools to investigate stability and instability of queueing networks. However, at the same time, it also has the following limitations: 
\begin{itemize}
\item Additional equations are needed for representing a particular service discipline and to describe such equations is sometimes difficult. For example, a non-preemptive priority discipline cannot be represented in an ordinary framework of fluid model. 
\item Fluid models are defined using only the first-order parameters such as the mean inter-arrival times and  mean service times, and hence we cannot apply them to the case where the stability of a queueing network depends on more than just the first-order parameters (see, for example, Dai et al.~\cite{Dai04} and Tezcan \cite{Tezcan13}). 
\end{itemize} 

In this paper, avoiding these limitations, we exploit another approach for studying stability and instability of queueing networks. We consider a queueing network in which the behavior of servers in each station is governed by a Markov chain with a finite number of server states. Such a service model is called a Markovian service process (MSP in short; see, for example, Ozawa \cite{Ozawa04}) and MSPs can represent various service disciplines such as a non-preemptive priority service. 
The whole network is represented as a continuous-time version of Markov modulated reflecting random walk (MMRRW in short), which is a multidimensional skip-free reflecting random walk whose transition probabilities are modulated according to the state of a Markov chain with a finite state space (see Ozawa \cite{Ozawa12}, where it is called a reflecting random walk with a background process). 
Here we should note that, since the stochastic model representing the queueing network we consider is a Markov chain with a countable state space, stability of the network corresponds to positive recurrence of the Markov chain and instability of the network to transience of the Markov chain; we do not treat the case where the Markov chain is null recurrent. We, therefore, compatibly use those words in the paper. 
By the uniformization technique, we can obtain a discrete-time MMRRW corresponding to the continuous-time MMRRW and, by the results of Ozawa \cite{Ozawa12}, we can expect to derive conditions on which the network is stable and on which it is instable. 
However, it is not so easy to apply the results of Ozawa \cite{Ozawa12} to a general queueing network having many stations and hence, in this paper, we will focus on a two-station network depicted in Fig.~\ref{fig:twostation}, in order to take the first step to studying general queueing networks. 
In the two-station network, exogenous customers arriving at station 1 join queue 1 (Q$_1$) and, after completing service there, they visit station 2 and join queue 2 (Q$_2$); after completing service in station 2, the customers reenter station 2 with probability $p$ and join queue 3 (Q$_3$). Exogenous customers arriving at station 2 join Q$_3$ and, after completing service there, they visit station 1 and join queue 4 (Q$_4$). In each station, there is just one server that serves customers in two queues according to a certain service discipline. We assume that exogenous customers arrive according to independent Markovian arrival processes (MAPs in short) and that the service times of customers in each queue are subject to a phase-type distribution. 
This model has a remarkable feature that the nominal condition, which means that the total offered load of each station is less than one, is not sufficient for the model to be stable, and its variations have been dealt with in a lot of literature (see, for example, Bramson \cite{Bramson08}, Dai et al.~\cite{Dai04} and Kumar \cite{Kumar93}). Hence we think it is the best model for our aim. 
%

%\input{fig_twostation.tex}
%%%%%%%%%%%%%%%%%%%%%%%%%%%%%%%%%%%%%%%%%%%%%%%%%%%%%%%%%%%
%
% fig_twostation.tex
%
%%%%%%%%%%%%%%%%%%%%%%%%%%%%%%%%%%%%%%%%%%%%%%%%%%%%%%%%%%%
\begin{figure}[bht]
\begin{center}
\setlength{\unitlength}{0.8mm}
\begin{picture}(100,42)(0,0)
%
% Station 1
\thicklines
\put(2,15){\makebox(0,0){\normalsize $\lambda_1$}}
\put(5,15){\vector(1,0){7}}

\put(14,6){\makebox(0,0){\normalsize Q$_1$}}
\put(27,11){\makebox(0,0){\normalsize $\mu_1$}}
\put(12,20){\line(1,0){6}}
\put(12,10){\line(1,0){6}}
\put(18,10){\line(0,1){10}}
\multiput(22,15)(2,0){6}{\line(1,0){1}}
\put(36,15){\vector(1,0){15}}

\put(42,40){\makebox(0,0){\normalsize Q$_4$}}
%\put(22,33){\makebox(0,0){\small High priority}}
\put(27,26){\makebox(0,0){\normalsize $\mu_4$}}
\put(37,25){\line(1,0){6}}
\put(37,35){\line(1,0){6}}
\put(37,25){\line(0,1){10}}
\multiput(22,30)(2,0){6}{\line(1,0){1}}
\put(19,30){\vector(-1,0){15}}

\put(28,0){\makebox(0,0){\normalsize Station 1}}
\put(20,5){\line(0,1){35}}
\put(20,5){\line(1,0){15}}
\put(20,40){\line(1,0){15}}
\put(35,5){\line(0,1){35}}
%
% Station 2
\thicklines
\put(54,6){\makebox(0,0){\normalsize Q$_2$}}
%\put(75,12){\makebox(0,0){\small High priority}}
\put(67,18){\makebox(0,0){\normalsize $\mu_2$}}
\put(52,20){\line(1,0){6}}
\put(52,10){\line(1,0){6}}
\put(58,10){\line(0,1){10}}
\multiput(62,15)(2,0){6}{\line(1,0){1}}
\put(88,20){\makebox(0,0){\normalsize $p$}}
\put(99,11){\makebox(0,0){\normalsize $1-p$}}
\put(76,15){\vector(1,0){28}}
\put(91,15){\line(0,1){13}}
\put(91,28){\vector(-1,0){8}}

\put(99,32){\makebox(0,0){\normalsize $\lambda_3$}}
\put(95,32){\vector(-1,0){12}}

\put(82,40){\makebox(0,0){\normalsize Q$_3$}}
\put(67,33){\makebox(0,0){\normalsize $\mu_3$}}
\put(77,25){\line(1,0){6}}
\put(77,35){\line(1,0){6}}
\put(77,25){\line(0,1){10}}
\multiput(62,30)(2,0){6}{\line(1,0){1}}
\put(59,30){\vector(-1,0){15}}

\put(68,0){\makebox(0,0){\normalsize Station 2}}
\put(60,5){\line(0,1){35}}
\put(60,5){\line(1,0){15}}
\put(60,40){\line(1,0){15}}
\put(75,5){\line(0,1){35}}

\end{picture}
\caption{A two-station network}
\label{fig:twostation}
\end{center}
\end{figure}
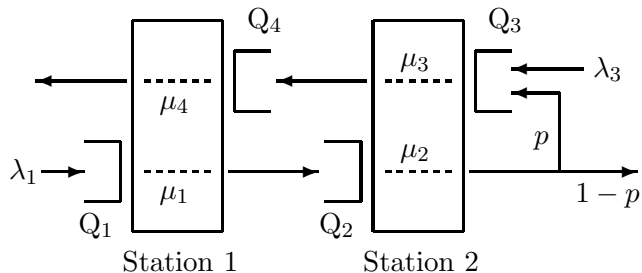
%%%%%%%%%%%%%%%%%%%%%%%%%%%%%%%%%%%%%%%%%%%%%%%%%%%%%%%%%%%

A key concept for investigating stability and instability of an MMRRW is {\it induced Markov chain}, introduced by Malyshev \cite{Malyshev93} and Fayolle et al.\ \cite{Fayolle95}. In our model, an induced Markov chain corresponds to an MMRRW obtained by supposing that some queues are saturated by customers, where the states of the saturated queues are removed from the original MMRRW. 
If the induced Markov chain is positive recurrent, we define the mean output rate of each queue as the expected number of customers departing from the queue per unit time in steady state and the mean input rate of each queue as the expected number of customers joining the queue per unit time in steady state. 
The conditions on which the two-station network is stable and on which it is instable are represented in terms of the difference between the mean input rate and mean output rate of each queue in each induced Markov chain that is positive recurrent. 
Since the mean input rates are given in terms of the mean output rates and the mean arrival rates of exogenous customers, the mean output rates are essential for determining the stability of the model. From this point, we can understand that, if all the mean output rates are given in terms of the first-order parameters, then the stability condition of the model is given in terms of the first-order parameters. Three examples we will consider in the paper fall into the case.

The rest of the paper is organized as follows. 
In Section \ref{sec:modelandresults}, the two-station network we consider is described in detail and 
the main results of the paper are stated with a typical example: a two-station network with a non-preemptive priority service. This section also functions as a summary of the paper. 
In Section \ref{sec:preliminary}, we make several preparations, where a Markov modulated reflecting random walk and the concept of induced Markov chain are introduced and related theorems are stated. 
In Section \ref{sec:positive}, we prove the main theorem using the theorems stated in Section \ref{sec:preliminary}. 
In Section \ref{sec:examples}, we consider two examples: one is a two-station network with a preemptive-resume priority service and the other that with a $(1,K)$-limited service. In the latter example, the stability region of the model depends on the value of parameter $K$.
%
%The paper is concluded with several remarks in Section \ref{sec:conclusion}. 

%%%%%%%%%%%%%%%%%%%%%%%
%
% Section 2
%
%%%%%%%%%%%%%%%%%%%%%%%
%
\section{Model description and main results} \label{sec:modelandresults}

\subsection{Two-station network}

We explain the two-station network depicted in Fig.~\ref{fig:twostation}, in detail.

\medskip
{\it Network configuration.}\quad 
Exogenous customers arrive at Q$_1$ as class-1 customers and at Q$_3$ as class-3 customers. After completing service at station~1, customers in Q$_1$ move to station~2 and join Q$_2$ as class-2 customers; customers in Q$_2$ next reenter station~2 and join Q$_3$ as class-3 customers with probability $p$ or depart from the system with probability $1-p$. After completing service at station~2, customers in Q$_3$ moves to station~1 and join Q$_4$ as class-4 customers; customers in Q$_4$ next depart from the system. 
We define a set $N$ as $N=\{1,2,3,4\}$ and, for $l\in N$, denote by $X_l(t)$ the number of customers in Q$_l$ including one being served at time $t$; $\bX(t)$ is the vector of $X_l(t)$, i.e., $\bX(t)=(X_1(t),X_2(t),X_3(t),X_4(t))$.
When $p=1$ and $\lambda_3=0$, the model is a usual reentrant line; when $p=0$, it is a two-station network with two-parallel-customer-flows. 

\medskip
{\it Arrival processes.}\quad 
The exogenous arrival process of class-1 customers and that of class-3 customers are subject to independent Markovian arrival processes (MAPs); MAPs are tractable arrival processes and they can represent correlated interarrival times. 
For $i=1,3$, we denote by $(\bar{C}_i,\bar{D}_i)$ the representation of the MAP for class-$i$ customers, where $\bar{C}_i$ is a matrix of phase transition rates without arrivals and $\bar{D}_i$ is that of phase transition rates with an arrival; the diagonal elements of $\bar{C}_i$ are set so that the sum of elements in each row of $\bar{C}_i+\bar{D}_i$ becomes zero (see, for example, Latouche and Ramaswami \cite{Latouche99}).  
We assume that the phase set of the MAP for class-$i$ customers is given by $S^a_i=\{1,2,...,s^a_i\}$, where $s^a_i$ is some positive integer; this $s^a_i$ is also the dimension of $\bar{C}_i$ and that of $\bar{D}_i$.
$\bar{C}_i+\bar{D}_i$ is the infinitesimal generator of the Markov chain that represents the phase process of the MAP for class-$i$ customers, and we assume it to be irreducible. Hence, $\bar{C}_i+\bar{D}_i$ is positive recurrent and we denote by $\bpi_i^*=(\pi_{ij}^*,j\in S^a_i)$ the stationary phase distribution. The mean arrival rate of the MAP for class-$i$ customers, denoted by $\lambda_i$, is given as $\lambda_i=\bpi_i^* \bar{D}_i\bone$, where $\bone$ is a column vector of $1$'s whose dimension is determined in context. 
For $i=1,3$, we denote by $J^a_i(t)$ the phase state of the MAP for class-$i$ customers at time $t$. 

\medskip
{\it Service processes.}\quad 
We assume that each station has a single server and the behavior of each server is represented as a two-class Markovian service process (MSP in short, see Ozawa \cite{Ozawa04}). For $i=1,2$, we denote by $J^s_i(t)$ the state of the server in station~$i$ at time $t$ and $S^s_i$ the set of states of the same server; we assume $S^s_i$ is finite. 
Furthermore, we assume that $S^s_1$ is composed of mutually disjoint subsets $S^s_{10}$, $S^s_{11}$ and $S^s_{14}$ and that $S^s_2$ is composed of mutually disjoint subsets $S^s_{20}$, $S^s_{22}$ and $S^s_{23}$, i.e., 
\[
S^s_1 = S^s_{10} \cup S^s_{11} \cup S^s_{14},\quad 
S^s_2 = S^s_{20} \cup S^s_{22} \cup S^s_{23}.
\]
When $J^s_1(t)\in S^s_{10}$ ($J^s_2(t)\in S^s_{20}$), the server in station~1 (resp.\ station~2) is idle or engaging in work other than service for customers; when $J^s_1(t)\in S^s_{11}$ ($J^s_2(t)\in S^s_{22}$), it is engaging in service for a class-1 (resp.\ class-2) customer; when $J^s_1(t)\in S^s_{14}$ ($J^s_2(t)\in S^s_{23}$), it is engaging in service for a class-4 (resp.\ class-3) customer. 
We assume the service process in station~1, $\{J^s_1(t)\}$, depends on $X_1(t)$ and $X_4(t)$ only through the events $\{X_1(t)=0\}$, $\{X_1(t)=1\}$, $\{X_1(t)\ge 2\}$, $\{X_4(t)=0\}$, $\{X_4(t)=1\}$ and $\{X_4(t)\ge 2\}$. 
Let $\bar{T}_1^{00}$ be the $|S^s_1|$-dimensional square matrix of transition rates for $\{J^s_1(t)\}$ without service completions when $X_1(t)=0$ and $X_4(t)=0$, where $|A|$ is the cardinality of set $A$. Also, let $\bar{T}_1^{+0}$ be the $|S^s_1|$-dimensional square matrix of transition rates for $\{J^s_1(t)\}$ without service completions when $X_1(t)\ge 1$ and $X_4(t)=0$; let $\bar{T}_1^{0+}$ be that without service completions when $X_1(t)=0$ and $X_4(t)\ge 1$; let $\bar{T}_1^{++}$ be that without service completions when $X_1(t)\ge 1$ and $X_4(t)\ge 1$. 
Let $\bar{T}_1^{1^*0}$ be the $|S^s_1|$-dimensional square matrix of transition rates for $\{J^s_1(t)\}$ with a service completion of class-1 customer when $X_1(t)=1$ and $X_4(t)=0$; in this case, the number of class-1 customers becomes zero after the service completion. Let $\bar{T}_1^{2^*0}$ be the $|S^s_1|$-dimensional square matrix of transition rates for $\{J^s_1(t)\}$ with a service completion of class-1 customer when $X_1(t)\ge 2$ and $X_4(t)=0$; in this case, the number of class-1 customers is still nonzero after the service completion. 
Let $\bar{T}_1^{01^*}$ be the $|S^s_1|$-dimensional square matrix of transition rates for $\{J^s_1(t)\}$ with a service completion of class-4 customer when $X_1(t)=0$ and $X_4(t)=1$; let $\bar{T}_1^{02^*}$ be that with a service completion of class-4 customer when $X_1(t)=0$ and $X_4(t)\ge 2$. 
Let $\bar{T}_1^{1^*+}$ be the $|S^s_1|$-dimensional square matrix of transition rates for $\{J^s_1(t)\}$ with a service completion of class-1 customer when $X_1(t)=1$ and $X_4(t)\ge 1$; let $\bar{T}_1^{2^*+}$ be that with a service completion of class-1 customer when $X_1(t)\ge 2$ and$X_4(t)\ge 1$. 
We analogously define $\bar{T}_1^{+1^*}$ and $\bar{T}_1^{+2^*}$. Note that $\bar{T}_1^{00}$ and the following sums of the matrices are infinitesimal generators, which we do not assume to be irreducible. 
\begin{align*}
&\bar{T}_1^{+0} + \bar{T}_1^{1^*0},\quad \bar{T}_1^{+0} + \bar{T}_1^{2^*0},\quad 
\bar{T}_1^{0+} + \bar{T}_1^{01^*},\quad \bar{T}_1^{0+} + \bar{T}_1^{02^*},\cr
&\bar{T}_1^{++} + \bar{T}_1^{1^*+} + \bar{T}_1^{+1^*},\quad 
\bar{T}_1^{++} + \bar{T}_1^{1^*+} + \bar{T}_1^{+2^*},\quad 
\bar{T}_1^{++} + \bar{T}_1^{2^*+} + \bar{T}_1^{+1^*},\cr 
&\bar{T}_1^{++} + \bar{T}_1^{2^*+} + \bar{T}_1^{+2^*}.
\end{align*}
Let $U_1^{0^*0}$ be the transition probability matrix for $\{J^s_1(t)\}$ when $X_1(t)=0$, $X_4(t)=0$ and a class-1 customer arrived at time $t$; let $U_1^{00^*}$ be that when $X_1(t)=0$, $X_4(t)=0$ and a class-4 customer arrived. 
Let $U_1^{+^*0}$ be the transition probability matrix for $\{J^s_1(t)\}$ when $X_1(t)\ge 1$, $X_4(t)=0$ and a class-1 customer arrived at time $t$; let $U_1^{+0^*}$ be that when $X_1(t)\ge 1$, $X_4(t)=0$ and a class-4 customer arrived. 
We analogously define $U_1^{0^*+}$, $U_1^{0+^*}$, $U_1^{+^*+}$ and $U_1^{++^*}$. 
With respect to station~2, we analogously define $\bar{T}_2^{ij}$ for $i,j=0,+$, $\bar{T}_2^{1^*0}$, $\bar{T}_2^{2^*0}$, $\bar{T}_2^{01^*}$, $\bar{T}_2^{02^*}$, $\bar{T}_2^{1^*+}$, $\bar{T}_2^{2^*+}$, $\bar{T}_2^{+1^*}$ and $\bar{T}_2^{+2^*}$, where we replace class-1 with class-3 and class-4 with class-2; we also define $U_1^{i^*j}$ and $U_1^{ij^*}$ for $i,j=0,+$. 

\medskip
{\it An example of service process: non-preemptive priority service.}\quad 
We consider the model in which class-4 customers have non-preemptive priority over class-1 customers and class-2 customers have non-preemptive priority over class-3 customers. 
For $i\in N$, we assume the service times of class-$i$ customers are subject to a phase-type (PH) distribution with representation $(\bbeta_i,\bar{H}_i)$, where $\bbeta_i$ and $\bar{H}_i$ are respectively the initial distribution and infinitesimal generator (restricted to transient states) of an absorbing Markov chain, by which the PH distribution is defined (see, for example, Latouche and Ramaswami \cite{Latouche99}). 
We define $\bh_i$ as $\bh_i=-\bar{H}_i \bone$ and denote by $\mu_i$ the reciprocal of the mean service time of class-$i$ customers; $\mu_i$ is given by $\mu_i=\left\{\bbeta_i (-\bar{H}_i)^{-1} \bone\right\}^{-1}$. 
PH distributions are known to be dense in the space of distributions on $[0,\infty)$.  
Let the phase set of the PH-distribution for class-1 customers be given by $S^s_{11}=\{2,3,...,s^s_1+1\}$ and that for class-4 customers by $S^s_{14}=\{s^s_1+2, s^s_1+3,...,s^s_1+s^s_4+1\}$, where $s^s_1$ and $s^s_4$ are some positive integers. 
Let the phase set of the PH-distribution for class-3 customers be given by $S^s_{23}=\{2,3,...,s^s_3+1\}$ and that for class-2 customers by $S^s_{22}=\{s^s_3+2, s^s_3+3,...,s^s_3+s^s_2+1\}$, where $s^s_2$ and $s^s_3$ are some positive integers.
Then, letting $S^s_{10}=S^s_{20}=\{1\}$, we have $S^s_1=\{1,2,...,s^s_1+s^s_4+1\}$ and $S^s_2=\{1,2,...,s^s_3+s^s_2+1\}$. 
A representation of the MSP in station~1 is given in block form as follows. 
\[
\bar{T}_1^{00} = \begin{pmatrix} 
0 & \bzero^\top & \bzero^\top \cr 
\bone & -I & O 
\cr \bone & O & -I \end{pmatrix},\quad 
\bar{T}_1^{+0} = \begin{pmatrix} 
-1 & \bbeta_1 & \bzero^\top \cr 
\bzero & \bar{H}_1 & O \cr 
\bzero & \bone \bbeta_1 & -I \end{pmatrix},\quad 
\bar{T}_1^{1^*0} = \begin{pmatrix} 
0 & \bzero^\top & \bzero^\top \cr 
\bh_1 & O & O \cr 
\bzero & O & O \end{pmatrix},
\]
\[ 
\bar{T}_1^{2^*0} = \begin{pmatrix} 
0 & \bzero^\top & \bzero^\top \cr 
\bzero & \bh_1 \bbeta_1 & O \cr 
\bzero & O & O \end{pmatrix},\quad 
\bar{T}_1^{0+} = \begin{pmatrix} 
-1 & \bzero^\top & \bbeta_4 \cr 
\bzero & -I & \bone \bbeta_4 \cr 
\bzero & O & \bar{H}_4 \end{pmatrix},\quad 
\bar{T}_1^{01^*} = \begin{pmatrix} 
0 & \bzero^\top & \bzero^\top \cr 
\bzero & O & O \cr 
\bh_4 & O & O \end{pmatrix},
\]
\[
\bar{T}_1^{02^*} = \begin{pmatrix} 
0 & \bzero^\top & \bzero^\top \cr 
\bzero & O & O \cr 
\bzero & O & \bh_4 \bbeta_4 \end{pmatrix},\quad 
\bar{T}_1^{++} = \begin{pmatrix} 
-1 & \bzero^\top & \bbeta_4 \cr 
\bzero & \bar{H}_1 & O \cr 
\bzero & O & \bar{H}_4 \end{pmatrix},
\]
\[
\bar{T}_1^{1^*+} = \bar{T}_1^{2^*+} = \begin{pmatrix} 
0 & \bzero^\top & \bzero^\top \cr 
\bzero & O & \bh_1 \bbeta_4 \cr 
\bzero & O & O \end{pmatrix},\quad 
\bar{T}_1^{+1^*} = \begin{pmatrix} 
0 & \bzero^\top & \bzero^\top \cr 
\bzero & O & O \cr 
\bzero & \bh_4 \bbeta_1 & O \end{pmatrix},
\]
\[
\bar{T}_1^{+2^*} = \begin{pmatrix} 
0 & \bzero^\top & \bzero^\top \cr 
\bzero & O & O \cr 
\bzero & O & \bh_4 \bbeta_4 \end{pmatrix},\quad
U_1^{0^*0} = \begin{pmatrix} 
0 & \bbeta_1 & \bzero^\top \cr 
\bzero & \bone \bbeta_1 & O \cr 
\bzero & \bone \bbeta_1 & O \end{pmatrix},\quad 
U_1^{00^*} = \begin{pmatrix} 
0　& \bzero^\top & \bbeta_4 \cr 
\bzero & O & \bone \bbeta_4 \cr 
\bzero & O & \bone \bbeta_4 \end{pmatrix},
\]
where $I$ is the identity matrix whose dimension is determined in context, $O$ a matrix of zeros whose dimension is determined in context and $\bzero$ a column vector of zeros whose dimension is determined in context; $U_1^{+^*0}$, $U_1^{+0^*}$, $U_1^{0^*+}$, $U_1^{0+^*}$, $U_1^{+^*+}$ and $U_1^{++^*}$ become identity matrices. 
Note that the matrices above have several dummy elements; for example, $\bar{T}_1^{00}$ has $\bone$'s and $-I$'s as blocks of dummy elements, which make a transition of $J_1^s(t)$ from a state of nonzero to the state of zero.
The representation of the MSP in station~2 are analogously given.

\medskip
{\it Markov chain.}\quad 
Define a vector $\bJ(t)$ as $\bJ(t)=(J^a_1(t),J^a_3(t),J^s_1(t),J^s_2(t))$, then the stochastic process $\{\bY(t)\}=\{(\bX(t),\bJ(t))\}$ representing the behavior of the two-station network becomes a continuous-time Markov chain on the state space $\calS = \mathbb{Z}_+^4 \times S_0$, where $S_0 = S^a_1 \times S^a_3 \times S^s_1 \times S^s_2$ and $\mathbb{Z}_+$ is the set of all non-negative integers. 
We denote by $Q=(Q(\bx,\bx'),\,\bx,\bx'\in\mathbb{Z}_+^4)$ the infinitesimal generator of the Markov chain $\{\bY(t)\}$, where, for each $\bx,\bx'\in\mathbb{Z}_+^4$, $Q(\bx,\bx')$ is a square block with the dimension of $s^a_1\,s^a_3\,|S^s_1|\,|S^s_2|$. Since the expression of $Q$ is very lengthy, we put it in Appendix \ref{sec:qmatrix}; in this paper, we do not directory use the expression in studying stability of the model.
We do not assume the Markov chain $\{\bY(t)\}$ is irreducible, but assume it has just one irreducible class; this point will be explained in the next subsection in detail. 
The Markov chain $\{\bY(t)\}$ is also a four-dimensional skip-free Markov modulated reflecting random walk (MMRRW in short, see Ozawa \cite{Ozawa12}, where it is called a reflecting random walk with a background process).

\subsection{Main results}

Before stating our main theorem, we make several preparations. 
In the previous subsection, we defined the Markov chain $\{\bY(t)\}=\{(\bX(t),\bJ(t))\}$ and assumed it was not necessary irreducible. This point is essential for some two-station networks; for example, in the two-station network with a non-preemptive priority service, it can be seen from some consideration that, after some finite time, the number of customers in Q$_2$, $X_2(t)$, and that of customers in Q$_4$, $X_4(t)$, do not take positive integers greater than one simultaneously. 
Hence we define a class of reducible Markov chain and call a Markov chain in the class a semi-irreducible Markov chain.

\begin{definition}[Semi-irreducible Markov chain] 
A continuous-time or discrete-time homogenous Markov chain on a countable set is said to be semi-irreducible if the chain has just one irreducible class, which is a closed communicating class, and every state in the irreducible class is accessible from any state of the chain. We consider that an irreducible Markov chain is also semi-irreducible. 
\end{definition}

\begin{remark} \label{re:semiirreducible}
If a Markov chain has a state that is accessible from any state of the Markov chain, then the Markov chain is semi-irreducible and the state is an element of the irreducible class. Therefore, it is not so difficult to check up on semi-irreducibility of a given Markov chain arising from a queueing model. 
For example, in the two-station network with a non-preemptive priority service described in the previous subsection, the state $(\bx,\bj)=((0,0,0,0),(1,1,1,1))\in\calS$ is accessible from any state in $\calS$ and hence the Markov chain $\{\bY(t)\}$ is semi-irreducible. 
\end{remark}

Let $\{Z(t)\}$ be a continuous-time homogenous Markov chain on a countable set $\tilde{\calS}$ and assume it is semi-irreducible with the irreducible class $\tilde{\calS}_0$.
With respect to {\it stability}, the Markov chain $\{Z(t)\}$ can be classified as follows.
\begin{itemize}
\item[(a)] The irreducible class is positive recurrent and the first passage time from any state in $\tilde{\calS}\setminus\tilde{\calS}_0$ to the irreducible class has finite expectation. 
\item[(b)] The irreducible class is positive recurrent and the first passage time from some state in $\tilde{\calS}\setminus\tilde{\calS_0}$ to the irreducible class does not have finite expectation. 
\item[(c)] The irreducible class is null recurrent. 
\item[(d)] The irreducible class is transient; this implies that every state of the Markov chain is transient. 
\end{itemize}
Since classes (a) and (d) are important in analyzing queueing models and it is difficult to study classes (b) and (c) by a standard method, we will focus only on classes (a) and (d) and use the following definition of positive recurrence and transience of semi-irreducible Markov chain.

\begin{definition}[Positive recurrence and transience of semi-irreducible Markov chain] \label{def:positive}
The semi-irreducible Markov chain $\{Z(t)\}$ is said to be positive recurrent if and only if it belongs to class (a), and it is said to be transient if and only if it belongs to class (d). 
\end{definition}
We also apply Definition \ref{def:positive} to discrete-time semi-irreducible Markov chains. 
If the semi-irreducible Markov chain $\{Z(t)\}$ is positive recurrent, the irreducible class $\tilde{\calS}_0$ has the stationary distribution, which we denote by $\tilde{\bpi}_0=(\tilde{\pi}_0(i),i\in\tilde{\calS}_0)$, and distribution $\tilde{\bpi}=(\tilde{\pi}(i),i\in\tilde{\calS})$ defined by
\[
\tilde{\pi}(i)= \left\{ \begin{array}{ll} 
\tilde{\pi}_0(i), & i\in\tilde{\calS}_0, \cr 
0, & \mbox{otherwise}, \end{array} \right.
\]
becomes the unique stationary distribution of the Markov chain $\{Z(t)\}$. Furthermore, the state distribution of $\{Z(t)\}$ converges to the stationary distribution $\tilde{\bpi}$ as $t$ tends to infinity. 
The positive recurrent semi-irreducible Markov chain $\{Z(t)\}$ is also positive Harris recurrent (see, for example, Bramson \cite{Bramson08}).
It can be seen from what we mentioned above that semi-irreducible Markov chains are very similar to irreducible Markov chains. 
Hereafter, we assume the Markov chain $\{\bY(t)\}$ is semi-irreducible and denote by $\calS_0$ its irreducible class. 
Furthermore, we assume there is at least one state in which the system is empty and that is accessible from any other states, i.e., for some $\bj\in S_0$, $((0,0,0,0),\bj)\in\calS_0$; this assumption is expected to hold for many cases and it implies that, for every $x_1\in\mathbb{Z}_+$ and for some $\bj\in S_0$, $((x_1,0,0,0),\bj)\in\calS_0$. 
We will use the point in the proof of Corollary \ref{co:stability} in Section \ref{sec:positive}.

%%%%%%%%%%%%%%%%%%%%%%%%%
Let $A$ be a non-empty subset of $N$ and denote by $\{\bY^A(t)\}=\{(\hat{\bX}^{N\setminus A}(t),\hat{\bJ}(t))\}$ the process obtained from $\{\bY(t)\}$ by assuming that, for every $i\in A$, Q$_i$ is saturated by customers, where $\hat{\bX}^{N\setminus A}(t)=(\hat{X}_i(t),i\in N\setminus A)$ and $\hat{\bJ}(t)=(\hat{J}_1^a(t),\hat{J}_3^a(t),\hat{J}_1^s(t),\hat{J}_2^s(t))$; since $\hat{X}_i(t)=\infty$ for $i\in A$, we removed them from the state vector $\hat{\bX}(t)=(\hat{X}_i(t),\,i\in N)$. 
We assume that the Markov chain $\{\bY^A(t)\}$ is semi-irreducible or that it has no irreducible class and every state is transient. If the Markov chain $\{\bY^A(t)\}$ is semi-irreducible and positive recurrent, we denote by $\bpi^A=(\bpi^A(\bx^{N\setminus A}),\,\bx^{N\setminus A}\in\mathbb{Z}_+^{4-|A|})$ its stationary distribution, where $\bx^{N\setminus A}=(x_i,\,i\in N\setminus A)$, $\bpi^A(\bx^{N\setminus A})=(\pi(\bx^{N\setminus A},\bj),\,\bj\in S_0)$ and $\pi(\bx^{N\setminus A},\bj)=P(\bY^A=(\bx^{N\setminus A},\bj))$. 
Furthermore, if the Markov chain $\{\bY^A(t)\}$ is positive recurrent, we define the mean input rate and the mean output rate of Q$_i$ in the process $\{\bY^A(t)\}$ as the mean number of customers arriving at Q$_i$ in a unit time and that of customers departing from Q$_i$ in a unit time, respectively; we denote the former by $\bar{\lambda}_i^A$ and the latter by $\bar{\mu}_i^A$.  
Since the expressions of the mean output rates are lengthy, we put them in Appendix \ref{sec:output}; in this paper, we do not directory use the expressions in studying stability of the model. The mean input rates are given by
\begin{equation}
\bar{\lambda}_1^A=\lambda_1,\quad 
\bar{\lambda}_2^A=\bar{\mu}_1^A,\quad 
\bar{\lambda}_3^A=\lambda_3+p \bar{\mu}_2^A,\quad 
\bar{\lambda}_4^A=\bar{\mu}_3^A. 
\label{eq:inputrate}
\end{equation}
We denote by $\Delta q_i^A$ the difference between the mean input rate and the mean output rate of Q$_i$ in the process $\{\bY^A(t)\}$, i.e., $\Delta q_i^A=\bar{\lambda}_i^A-\bar{\mu}_i^A$. Note that, if $i\in N\setminus A$, then $\Delta q_i^A=0$ since the input rate of Q$_i$ balances with the corresponding output rate under the assumption that the process $\{\bY^A(t)\}$ is positive recurrent. 

\medskip
For $i\in N$, denote by $h_i$ the mean service time of class-$i$ customers, then the nominal condition of the two-station network is given by 
\begin{equation}
\rho_1+\rho_4<1,\quad 
\rho_2+\rho_3<1, 
\label{eq:nominal}
\end{equation}
where $\rho_1=\lambda_1 h_1$, $\rho_2=\lambda_1 h_2$, $\rho_3=(p\lambda_1+\lambda_3) h_3$ and $\rho_4=(p\lambda_1+\lambda_3) h_4$. 
In this paper, since we are interested in the case where the nominal condition is not sufficient for the two-station network to be stable, we focus on the model satisfying the following conditions: this point will be clear later. 
\begin{align}
& \Delta q_1^N>0,\quad \Delta q_2^N<0,\quad \Delta q_3^N>0,\quad \Delta q_2^N<0, \label{eq:dqN} \\
& \Delta q_1^{\{1,2,3\}}<0,\quad \Delta q_2^{\{1,2,3\}}>0,\quad \Delta q_3^{\{1,2,3\}}>0, \label{eq:dq123} \\
& \Delta q_1^{\{1,3,4\}}>0,\quad \Delta q_3^{\{1,3,4\}}<0,\quad \Delta q_4^{\{1,3,4\}}>0, \label{eq:dq134} \\
& \Delta q_1^{\{1,4\}}>0,\quad \Delta q_4^{\{1,4\}}<0, \label{eq:dq14} \\
& \Delta q_2^{\{2,3\}}<0,\quad \Delta q_3^{\{2,3\}}>0. \label{eq:dq23}
\end{align}
Note that, since $\{\bY^N(t)\}$ is a finite semi-irreducible Markov chain, it is always positive recurrent and $\Delta q_i^N$ for $i\in N$ exists; Under condition (\ref{eq:dqN}), both $\{\bY^{\{1,2,3\}}(t)\}$ and $\{\bY^{\{1,3,4\}}(t)\}$ are positive recurrent (the reason will be explained in Section \ref{sec:positive}) and $\Delta q_i^{\{1,2,3\}},\,i=1,2,3,$ and $\Delta q_i^{\{1,3,4\}},\,i=1,3,4,$ exist; Under conditions (\ref{eq:dqN}) to (\ref{eq:dq134}), both $\{\bY^{\{1,4\}}(t)\}$ and $\{\bY^{\{23\}}(t)\}$ are positive recurrent (the reason will also be explained in Section \ref{sec:positive}) and $\Delta q_i^{\{1,4\}},\,i=1,4,$ and $\Delta q_i^{\{2,3\}},\,i=2,3,$ exist.

Sufficient conditions on which the semi-irreducible Markov chain $\{\bY(t)\}$ is positive recurrent and it is transient are given by the following theorem. 
\begin{theorem} \label{th:main0}
Let $r_1$ and $r_2$ be defined as 
\begin{align}
&r_1 = \frac{\Delta q_2^{\{1,2,3\}} \Delta q_3^{\{2,3\}} - \Delta q_3^{\{1,2,3\}} \Delta q_2^{\{2,3\}}}{\Delta q_1^{\{1,2,3\}} \Delta q_2^{\{2,3\}}},\quad 
r_2 = \frac{\Delta q_4^{\{1,3,4\}} \Delta q_1^{\{1,4\}} - \Delta q_1^{\{1,3,4\}} \Delta q_4^{\{1,4\}}}{\Delta q_3^{\{1,3,4\}} \Delta q_4^{\{1,4\}}}, 
\end{align}
and assume that 
\begin{equation}
\left| \frac{\Delta q_1^N}{\Delta q_4^N} \right| \le \left| \frac{\Delta q_1^{\{1,4\}}}{\Delta q_4^{\{1,4\}}} \right|\quad \mbox{and}\quad
\left| \frac{\Delta q_3^N}{\Delta q_2^N} \right| < \left| \frac{\Delta q_3^{\{2,3\}}}{\Delta q_2^{\{2,3\}}} \right| 
\label{eq:main0cond1}
\end{equation}
or that 
\begin{equation}
\left| \frac{\Delta q_1^N}{\Delta q_4^N} \right| < \left| \frac{\Delta q_1^{\{1,4\}}}{\Delta q_4^{\{1,4\}}} \right| \quad \mbox{and}\quad \left| \frac{\Delta q_3^N}{\Delta q_2^N} \right| \le \left|\frac{\Delta q_3^{\{2,3\}}}{\Delta q_2^{\{2,3\}}} \right|. 
\label{eq:main0cond2}
\end{equation}
Then, the semi-irreducible Markov chain $\{\bY(t)\}$ is positive recurrent if $r_1r_2<1$ and it is transient if $r_1r_2>1$.
\end{theorem}

%%%%%%%%%%%%%
%\medskip
{\it An example: the two-station network with a non-preemptive priority service.}\quad Here we demonstrate how the theorem works in the case of the two-station network with a non-preemptive priority service explained in the previous subsection. 
For $i\in N$, the service times of class-$i$ customers are subject to a PH-distribution with mean service time $1/\mu_i$ and class-4 customers (class-2 customers) have non-preemptive priority over class-1 customers (resp.\ class-3 customers). 
Assuming the nominal condition, we consider the case where $\mu_1>\mu_2$ and $\mu_3>\mu_4$; in this case, the nominal condition is not sufficient for the model to be stable. 
Let a set $\calN_p$ be defined as $\calN_p=\{N,\{1,2,3\},\{1,3,4\},\{1,4\},\{2,3\}\}$. For $A\in\calN_p$ and for $i\in A$, $\Delta q_i^A$ is given as follows.
\begin{itemize}
\item $\Delta q_i^N,\,i\in N$. 
Consider the case where all the queues are saturated by customers. Since class-4 customers have non-preemptive priority over class-1 customers, the server in station 1 is engaged in service only for class-4 customers after some finite time. Hence, the mean output rates of Q$_1$ and Q$_4$ are given as $\bar{\mu}_1^N=0$ and $\bar{\mu}_4^N=\mu_4$, respectively. Analogously, we obtain $\bar{\mu}_2^N=\mu_2$ and $\bar{\mu}_3^N=0$. 
The mean input rates are automatically obtained by formulae (\ref{eq:inputrate}) and we have 
\[
\Delta q_1^N=\lambda_1>0,\quad 
\Delta q_2^N=-\mu_2<0,\quad 
\Delta q_3^N=\lambda_3+p \mu_2>0,\quad 
\Delta q_4^N=-\mu_4<0.
\]
\item $\Delta q_i^{\{1,2,3\}},\,i\in {\{1,2,3\}}$.  
Consider the case where queues Q$_1$, Q$_2$ and Q$_3$ are saturated by customers. In this case, the server in station 2 is engaged in service only for class-2 customers after some finite time, and we have $\bar{\mu}_2^{\{1,2,3\}}=\mu_2$ and $\bar{\mu}_4^{\{1,2,3\}}=0$. This implies that any customers do not arrive at Q$_4$ after some finite time and Q$_4$ will eventually become empty since customers in Q$_4$ have priority in service over customers in Q$_1$. Hence we obtain $\bar{\mu}_4^{\{1,2,3\}}=0$. 
After Q$_4$ becomes empty, the server in station 1 is engaged in service for customers in Q$_1$, which is saturated by customers. Hence, we have $\bar{\mu}_1^{\{1,2,3\}}=\mu_1$. As a result, we obtain 
\[
\Delta q_1^{\{1,2,3\}}=\lambda_1-\mu_1<0,\quad 
\Delta q_2^{\{1,2,3\}}=\mu_1-\mu_2>0,\quad 
\Delta q_3^{\{1,2,3\}}=\lambda_3+p \mu_2>0, 
\]
where we use the assumption of  $\mu_1>\mu_2$.
\item $\Delta q_i^{\{1,3,4\}},\,i\in {\{1,3,4\}}$.  
Consider the case where queues Q$_1$, Q$_3$ and Q$_4$ are saturated by customers. This case is symmetric to that for $\Delta q_i^{\{1,2,3\}}$. Hence, we obtain 
\[
\Delta q_1^{\{1,3,4\}}=\lambda_1>0,\quad 
\Delta q_3^{\{1,3,4\}}=\lambda_3-\mu_3<0,\quad 
\Delta q_4^{\{1,3,4\}}=\mu_3-\mu_4>0, 
\]
where we use the assumption of $\mu_3>\mu_4$.
\item $\Delta q_i^{\{1,4\}},\,i\in {\{1,4\}}$.  
Consider the case where queues Q$_1$ and Q$_4$ are saturated by customers. In this case, the server in station 1 is engaged in service only for class-4 customers after some finite time, and we have $\bar{\mu}_1^{\{1,4\}}=0$ and $\bar{\mu}_4^{\{1,4\}}=\mu_4$. This implies that any customers do not arrive at Q$_2$ after some finite time and Q$_2$ will eventually become empty. Hence, we have $\bar{\mu}_2^{\{1,4\}}=0$. 
After Q$_2$ becomes empty, the server in station 2 is engaged in service for customers in Q$_3$. From the nominal condition (\ref{eq:nominal}), we have $\lambda_3 < \mu_3$, and hence we see that Q$_3$ is stable and the output rate of Q$_3$ is identical to its input rate, i.e., $\bar{\mu}_3^{\{1,4\}}=\lambda_3$. As a result, we have 
\[
\Delta q_1^{\{1,4\}}=\lambda_1>0,\quad 
\Delta q_4^{\{1,4\}}=\lambda_3-\mu_4<0, 
\]
where we use the nominal condition to derive the negativity of $\Delta q_4^{\{1,4\}}$. 
\item $\Delta q_i^{\{2,3\}},\,i\in {\{2,3\}}$.  
Consider the case where queues Q$_2$ and Q$_3$ are saturated by customers. This case is symmetric to that for $\Delta q_i^{\{1,4\}}$. Hence, we have 
\[
\Delta q_2^{\{2,3\}}=\lambda_1-\mu_2<0,\quad 
\Delta q_3^{\{2,3\}}=\lambda_3+p \mu_2>0, 
\]
where we use the nominal condition to derive the negativity of $\Delta q_2^{\{2,3\}}$. 
\end{itemize}

We note that all $\Delta q_i^A,\,A\in\calN_p,\,i\in A,$ satisfies conditions (\ref{eq:dqN}) to (\ref{eq:dq23}). 
Furthermore, the following conditions of Theorem \ref{th:main0} are also satisfied: 
\[
\left|\frac{\Delta q_1^N}{\Delta q_4^N}\right|=\frac{\lambda_1}{\mu_4} \le \frac{\lambda_1}{\mu_4-\lambda_3} = \left|\frac{\Delta q_1^{\{1,4\}}}{\Delta q_4^{\{1,4\}}}\right|, \quad 
\left|\frac{\Delta q_3^N}{\Delta q_2^N}\right|=\frac{\lambda_3+p \mu_2}{\mu_2} < \frac{\lambda_3+p \mu_2}{\mu_2-\lambda_1}=\left|\frac{\Delta q_3^{\{2,3\}}}{\Delta q_2^{\{2,3\}}}\right|, 
\]
where we consider the case where $\lambda_1\ne 0$ and $\lambda_3+p \mu_2\ne 0$. 
$r_1$ and $r_2$ in Theorem \ref{th:main0} are given as 
\[
r_1 = \frac{\lambda_3+p \mu_2}{\mu_2-\lambda_1},\quad 
r_2 = \frac{\lambda_1}{\mu_4-\lambda_3}, 
\]
and inequality $r_1 r_2<1$ is equivalent to $\rho_2+\rho_3<1$. Hence, we see that the two-station network with a non-preemptive priority service is positive recurrent if $\rho_2+\rho_3<1$ and it is transient if $\rho_2+\rho_3>1$; This result is coincident with the existing result for two-station networks with a preemptive-resume priority service. 

\begin{remark}
It can be seen from the above results for the two-station network with a non-preemptive priority service that the condition on which the model is stable and it is unstable are given only in terms of the mean arrival rates and the mean service rates; hence, the stability region of the model does not depend on other features of the arrival processes and service time distributions. 
On the other hand, Dai et al.~\cite{Dai04} demonstrated that the stability region of a push-started Lu-Kumar network depended on the inter-arrival time distribution and service time distributions. Our two-station network is a simplified version of the push-started Lu-Kumar network, and it can be seen from the results in Dai et al.~\cite{Dai04} that, in the case of constant inter-arrival times and constant service times and in the case of uniformly distributed inter-arrival times and uniformly distributed service times, the condition $\rho_2+\rho_4>1$ does not always imply instability of the two-station network with a non-preemptive priority service. 
Any arrival process with i.i.d.\ inter-arrival times can be approximated with any accuracy by a MAP and any service time distribution can also be approximated with any accuracy by a PH-distribution. However, any arrival process whose inter-arrival time distribution has bounded support cannot be represented as a MAP and any service time distribution with bounded support cannot also be represented as a PH-distribution. 
Therefore, as mentioned in Dai et al.~\cite{Dai04}, the boundedness of the support of the distributions seems to affect the stability of the network. 
\end{remark}

%%%%%%%%%%%%%%%%%%%%%%%
%
% Section 3
%
%%%%%%%%%%%%%%%%%%%%%%%
%
\section{Preliminaries} \label{sec:preliminary}

In this section, we deal with a discrete-time semi-irreducible Markov chain and Markov modulated reflecting random walk (MMRRW) and present several concepts and theorems that will be used for proving the main theorem. 
The model we considered in the previous section was a continuous-time semi-irreducible Markov chain; in the next section, applying the uniformization technique to the continuous-time semi-irreducible Markov chain, we will obtain a discrete-time semi-irreducible Markov chain that has the same stability region as that of the original continuous-time semi-irreducible Markov chain. Applying the theorems in this section to the discrete-time semi-irreducible Markov chain, we will prove the main theorem. 
Note that, in this section, we use notations independently of other sections.

\subsection{Semi-irreducible Markov chain}

Let $\{Z_n\}$ be a discrete-time homogenous Markov chain on a countable set $\calS$ and assume it to be semi-irreducible with irreducible class $\calS_0$. As mentioned in Section \ref{sec:modelandresults}, semi-irreducible Markov chains are very similar to irreducible Markov chains. 
If the semi-irreducible Markov chain $\{Z_n\}$ is positive recurrent in our sense and it has the stationary distribution denoted by $\bpi=(\pi(i),i\in\calS)$, then the ergodic theorem of Markov chains also holds, i.e., for a real function $f$ on $\calS$ satisfying 
\[
\sum_{i\in\calS} |f(i)| \pi(i) < \infty, 
\]
we have, for any initial distribution, 
\[
\lim_{n\to\infty} \frac{1}{n}\sum_{k=1}^n f(Z_n) = \sum_{i\in\calS} f(i) \pi(i),\ \mbox{w.p. 1}.
\]
In order to derive conditions on which a semi-irreducible Markov chain is positive recurrent and it is transient in our sense, we can use the following statements. 

\begin{theorem}[Foster's theorem] \label{th:Foster1}
The semi-irreducible Markov chain $\{Z_n\}$ is positive recurrent if there exist a positive number $\varepsilon$, a finite subset $\calV\subset\calS$ and a lower bounded real function $f(i),\,i\in\calS,$ such that 
\begin{align}
&E(f(Z_{n+1})-f(Z_n)\,|\,Z_n=i) \le -\varepsilon,\ i\in\calS\setminus\calV, \label{eq:Foster1a} \\
&E(f(Z_{n+1})\,|\,Z_n=i) < \infty,\ i\in\calV. \label{eq:Foster1b}
\end{align}
\end{theorem}

\begin{remark}
In applying Theorem \ref{th:Foster1} to a semi-irreducible Markov chain, it is not necessary to specify the irreducible class of the Markov chain. 
\end{remark}

Make sure that Theorem \ref{th:Foster1} asserts that, under conditions (\ref{eq:Foster1a}) and (\ref{eq:Foster1b}), not only the irreducible class $\calS_0$ is positive recurrent in a usual sense but also the first passage time from any state of the Markov chain to the irreducible class has finite expectation. Furthermore, under the conditions, we must have $\calV\cap\calS_0\ne\emptyset$. 
Theorem \ref{th:Foster1} can be proved in a manner similar to that for irreducible Markov chains (see, for example, Br\'emaud \cite{Bremaud99}), and we here omit the proof.

Exclusively divide the state space $\calS$ into a finite number of nonempty subsets, say $\calV_k,\,k=1,2,...,m_S$, and define a function $\tau(i)$ as 
\[
\mbox{$\tau(i) = T_k$ if $i\in\calV_k$}, 
\]
where, for $k\in\{1,2,...,m_S\}$, $T_k$ is a positive integer. 
Furthermore, define a strictly increasing time sequence $\{t_n\}$ as 
\[
t_0=0,\quad t_{n+1} = t_n+\tau(Z_{t_n}),\,n\ge 0,
\]
and consider a stochastic process $\{\tilde{Z}_n\}$ defined by $\tilde{Z}_n=Z_{t_n},\,n\ge 0$. This process $\{\tilde{Z}_n\}$ is a kind of embedded Markov chain of $\{Z_n\}$, and we immediately obtain the following corollary, which is a modification of Theorem 2.2.4 of Fayolle et al.\ \cite{Fayolle95}; also see Proposition 4.5 of Bramson \cite{Bramson08}.

\begin{corollary} \label{co:Foster2}
Assume that the embedded Markov chain $\{\tilde{Z}_n\}$ is semi-irreducible with the same irreducible class $\calS_0$ as the original chain $\{Z_n\}$. If there exist a positive number $\varepsilon$, a finite subset $\calV\subset\calS$ and a lower bounded real function $f(i),\,i\in\calS,$ such that  
\begin{align}
&E(f(\tilde{Z}_{n+1})-f(\tilde{Z}_n)\,|\,\tilde{Z}_n=i) \le -\varepsilon,\ i\in\calS\setminus\calV, \label{eq:Foster2a} \\
&E(f(\tilde{Z}_{n+1})\,|\,\tilde{Z}_n=i) < \infty,\ i\in\calV, \label{eq:Foster2b}
\end{align}
then the semi-irreducible Markov chain $\{Z_n\}$ is positive recurrent.
\end{corollary}

In order to prove that a semi-irreducible Markov chain is transient, we will use the following proposition, which is obtained by applying Theorem 2.2.7 of Fayolle et al.\ \cite{Fayolle95} to the irreducible class $\calS_0$. 
\begin{theorem} \label{th:Markov_transient} 
Assume that the embedded Markov chain $\{\tilde{Z}_n\}$ is semi-irreducible with the same irreducible class $\calS_0$ as the original chain $\{Z_n\}$. If there exist a real function $f(i),\,i\in\calS,$ and positive numbers $\varepsilon$, $c$ and $b$ such that, for $\calA=\{i\in\calS : f(i)>c \}$,   
\begin{itemize}
\item[(i)] $\calA\cap\calS_0\ne\emptyset,\ \calA^C\cap\calS_0\ne\emptyset$, 
\item[(ii)] $E(f(\tilde{Z}_{n+1})-f(\tilde{Z}_n)\,|\,\tilde{Z}_n=i) \ge \varepsilon,\ i\in\calA$, and
\item[(iii)] the inequality $|f(j)-f(i)|>b$ implies $P(Z_1=j\,|\,Z_0=i)=0$, 
\end{itemize}
then the semi-irreducible Markov chain $\{Z_n\}$ is transient.
\end{theorem}

\begin{remark}
In Theorem \ref{th:Markov_transient}, condition (ii) can be replaced with
\begin{itemize}
\item[(ii')] $E(f(\tilde{Z}_{n+1})-f(\tilde{Z}_n)\,|\,\tilde{Z}_n=i) \ge \varepsilon,\ i\in\calA\cap\calS_0$, 
\end{itemize}
but, in order to use this condition, we have to exactly specify the irreducible class $\calS_0$. Hence we will use condition (ii) instead of (ii') in this paper. 
\end{remark}

\subsection{Markov modulated reflecting random walk}

According to Ozawa \cite{Ozawa12}, we introduce a Markov modulated reflecting random walk and give a condition on which it is positive recurrent. 

Let $d$ be a positive integer and define a set $N$ as $N=\{1,2,...,d\}$. 
Hereafter, we use the following notations. For $A\subset N$, we denote by $\bx^A$ a vector in $\mathbb{Z}_+^{|A|}$ whose elements are indexed in $A$, i.e., $\bx^A=(x_l,\,l\in A)$, and by $\btwo^A$ the $|A|$-dimensional vector of $2$'s whose elements are indexed in $A$. We also denote by $(\bx^A,\by^{N\setminus A})$ the combination of $\bx^A$ and $\by^{N\setminus A}$; the $l$-th element of $(\bx^A,\by^{N\setminus A})$ is $x_l$ if $l\in A$ and $y_l$ if $l\in N\setminus A$. 
For example, when $d=5$ and $A=\{1,3\}$, we have $\bx^A=(x_1,x_3)$, $\by^{N\setminus A}=(y_2,y_4,y_5)$ and $(\bx^A,\by^{N\setminus A})=(x_1,y_2,x_3,y_4,y_5)$.
We also use this notation for expressing a part of a given vector; for $\bx=(x_l,l\in N)$ and $A\subset N$, we denote vector $(x_l,l\in A)$ by $\bx^A$ and represent $\bx$ as $\bx=(\bx^A,\bx^{N\setminus A})=(\bx^{N\setminus A},\bx^A)$. 

Define a function $\varphi:\,\mathbb{Z}_+^d\to\calN$ as
\begin{equation}
\varphi(\bx) = N\setminus\{l\in N : x_l=0 \}\ \mbox{for}\ \bx=(x_1,x_2,...,x_d)\in\mathbb{Z}_+^d, 
\label{eq:varphi}
\end{equation}
where $\calN$ is the set of all subsets of $N$; $\varphi(\bx)$ is the index set of the nonzero elements of $\bx$ and indicates the boundary face of $\mathbb{Z}_+^d$ that $\bx$ is belonging to.
For $A\in\calN$, let $S^A$ be a finite set, say $S^A=\{1,2,...,s^A\}$ for some positive integer $s^A$. 
Let a state space $\calS$ be defined as 
\[
\calS = \bigcup_{\bx\in\mathbb{Z}_+^d} \left( \{\bx\}\times S^{\varphi(\bx)} \right), 
\]
and consider a discrete-time homogeneous Markov chain $\{\bY_n\}=\{(\bX_n,J_n)\}$ on the state space $\calS$; we call the process $\{\bX_n\}=\{(X_{1,n},X_{2,n},\cdots,X_{d,n})\}$ the front process of $\{\bY_n\}$ and $\{J_n\}$ the background process of it. 
Furthermore, we assume that individual processes $\{X_{l,n}\},\,l\in N,$ are skip free and impose a kind of space-homogeneous structure of transition probabilities, as follows. For $A\in\calN$, let $\calB^A$ be a boundary face of $\calS$, defined as 
\begin{equation}
\calB^A = \{(\bx,i)\in\calS : \varphi(\bx)=A\}; 
\label{eq:boundary}
\end{equation}
$\calB^N$ is the interior space of $\calS$, but we also call it a boundary face. 
Transition probabilities are space-homogeneous in each boundary face; this means that, for $(\bx,i)\in\calS$, $\bz\in\{-1,0,1\}^d$ such that $\bx+\bz\in\mathbb{Z}_+^d$ and $j\in S^{\varphi(\bx+\bz)}$, the transition probabilities are given in the form of 
\begin{align}
P(\bY_{n+1}=(\bx+\bz,j)\,|\,\bY_{n+1}=(\bx,i)) = p_{\bz}^{\varphi(\bx)\,\varphi(\bx+\bz)}(i,j), 
\label{eq:trans_prob}
\end{align}
where we have, for $A=\varphi(\bx)\subset N$, 
\[
\sum_{\bz^{A}\in\{-1,0,1\}^{|A|}}\ \sum_{\bz^{N\setminus A}\in\{0,1\}^{d-|A|}}\ \sum_{j\in S^{\varphi(\bx+(\bz^A,\bz^{N\setminus A}))}} p_{(\bz^A,\bz^{N\setminus A})}^{A,\varphi(\bx+(\bz^{A},\bz^{N\setminus A}))}(i,j) =1.
\]
We call the process $\{\bY_n\}$ a $d$-dimensional skip-free Markov modulated reflecting random walk (MMRRW in short, it was called a reflecting random walk with a background process, in Ozawa \cite{Ozawa12}); we denote it by $\calL$. 
In Ozawa \cite{Ozawa12}, $\calL$ was assumed to be irreducible, but in this paper, we assume $\calL$ is semi-irreducible and denote by $\calS_0$ its irreducible class. 
Let $\bxi_n=(\xi_{1,n},\xi_{2,n},...,\xi_{d,n})$ be the vector of increments defined as $\bxi_n=\bX_n-\bX_{n-1}$. By the assumption of skip-free property, we have $\bxi_n\in {\{-1,0,1\}^d}$. For $\by\in\calS$, we denote by $\balpha(\by)$ the vector of the conditional mean increments given that the chain is in state $\by$, i.e., 
\begin{align*}
&\balpha(\by)= (\alpha_1(\by),\alpha_2(\by),\cdots,\alpha_d(\by)),\  
\alpha_l(\by) =E(\xi_{l,n+1}\,|\,\bY_n=\by),\ l\in N. 
\end{align*}

An important clue to understand stability of multidimensional MMRRWs is the concept of induced Markov chain, which was introduced by Malyshev \cite{Malyshev93} and Fayolle et al.\ \cite{Fayolle95}. 
Let $A$ be a non-empty subset of $N$. The induced Markov chain of $\calL$ with respect to $A$, denoted by $\calL^A$, is a $(d-|A|)$-dimensional MMRRW, $\{(\hat{\bX}^{N\setminus A}_n,\hat{J}_n)\}$, whose state space is $\calS^A$ defined as 
\begin{equation}
\calS^A = \bigcup_{\bx^{N\setminus A}\in\mathbb{Z}_+^{d-|A|}} \left( \{\bx^{N\setminus A}\}\times S^{\varphi(\bx^{N\setminus A},\btwo^A)} \right)
\end{equation}
and whose transition probabilities are given by, for $(\bx^{N\setminus A},i)\in\calS^A$ and $\bz^{N\setminus A}\in\{-1,0,1\}^{d-|A|}$ such that $\bx^{N\setminus A}+\bz^{N\setminus A}\in\mathbb{Z}_+^{d-|A|}$ and for $j\in S^{\varphi(\bx^{N\setminus A}+\bz^{N\setminus A},\btwo^A)}$, 
\begin{align}
&P\Big( (\hat{\bX}^{N\setminus A}_{n+1},\hat{J}_{n+1})=(\bx^{N\setminus A}+\bz^{N\setminus A},j)\,\big|\,(\hat{\bX}^{N\setminus A}_n,\hat{J}_n)=(\bx^{N\setminus A},i) \Big) \cr
&\qquad = \sum_{\bz^A\in\{-1,0,1\}^{|A|}} p_{(\bz^{N\setminus A},\bz^A)}^{\varphi(\bx^{N\setminus A},\btwo^A)\,\varphi(\bx^{N\setminus A}+\bz^{N\setminus A},\btwo^A)}(i,j). 
\end{align}
In Ozawa \cite{Ozawa12}, for all $A\subset N$ such that $A\ne\emptyset$, $\calL^A$ was assumed to be irreducible, but in this paper, we assume that it is semi-irreducible or that it has no irreducible class and every state in $\calS^A$ is transient. If $\calL^A$ is semi-irreducible and positive recurrent, then we denote its stationary distribution by $\bpi^A=(\pi^A(\bx^{N\setminus A},i),\,(\bx^{N\setminus A},i)\in\calS^A)$. Note that $\calL^N$ is a finite semi-irreducible Markov chain and it is always positive recurrent. 
For $\calL^A$ being positive recurrent, we define the $d$-dimensional mean increment vector, $\ba(A)=(a_l(A),\,l\in N)$, by 
\begin{equation}
a_l(A) = \sum_{(\bx^{N\setminus A},i)\in\calS^A} \alpha_l((\bx^{N\setminus A},\btwo^A),i)\,\pi^A(\bx^{N\setminus A},i),\ l\in N. 
\label{eq:aA}
\end{equation}
From the definition, it can be seen that $a_l(A) = 0$ for all $l\in N\setminus A$. The mean increment vectors play a crucial role in analyzing stability of the multidimensional MMRRW. 

\medskip
{\it An example of MMRRW.}\quad Consider a single-class $3$-station generalized Jackson network with Markovian arrivals and phase-type services. The behavior of this model can be represented as a continuous-time Markov chain $\{\bar{\bY}(t)\}=\{((\bar{X}_{1}(t),\bar{X}_{2}(t),\bar{X}_{3}(t)), \bar{J}(t))\}$, where $\bar{X}_i(t)$ is the number of customers in station $i$ at time $t$ and $\bar{J}(t)$ is the state (phase) of the process combining the Markovian arrival processes and phase-type service processes. 
This $\{\bar{\bY}(t)\}$ is a continuous-time version of $3$-dimensional MMRRW and, by the uniformization technique, we obtain a discrete-time $3$-dimensional MMRRW, $\{\bY_n\}=\{((X_{1,n},X_{2,n},X_{3,n}),J_n)\}$, denoted by $\calL$. 
We consider that the discrete-time MMRRW represents the behavior of a discrete-time $3$-station network corresponding to the original continuous-time network. 
For $A\subset N=\{1,2,3\}$ such that $A\ne\emptyset$, the induced Markov chain $\calL^A$ generated from $\calL$, $\{(\hat{\bX}^{N\setminus A}_n,\hat{J}_n)\}$, is a $(3-|A|)$-dimensional MMRRW that represents, for $i\in N\setminus A$, the process of the number of customers in station $i$ as well as the background process under the condition that, for $j\in A$, station $j$ is saturated by customers. 
For $i\in N\setminus A$, the $i$-th element of the mean increment vector $\ba(A)$ is the mean increment rate of the number of customers in station $i$; for $i\in A$, since station $i$ is considered to be saturated by customers, we regard the $i$-th element of $\ba(A)$ as the difference of the mean input rate (mean arrival rate) and the mean output rate (mean departure rate) of station $i$. 
For example, if $A=\{1,2\}$, then $\calL^{\{1,2\}}$ is the one-dimensional MMRRW, $\{(\hat{\bX}_{3,n},\hat{J}_n)\}$, where $\hat{\bX}_{3,n}$ is the number of customers in station 3 and $\hat{J}_n$ is the state of the background process under the condition that both stations 1 and 2 are saturated by customers. If $\calL^{\{1,2\}}$ is positive recurrent, we have $a_3(\{1,2\})=0$; $a_1(\{1,2\})$ is given by the mean arrival rate of station 1 minus the mean departure rate of the same station and $a_2(\{1,2\})$ is analogously given. 
We see from this example that it is rather easy to understand the induced Markov chains of a multidimensional MMRRW arising from a queueing network. 

\medskip
Assume that each induced Markov chain is positive recurrent or transient, and let $\calN_p$ be the index set of positive-recurrent induced Markov chains, i.e., 
\[
\calN_p = \{A\subset N: \mbox{$A\ne\emptyset$ and $\calL^A$ is positive recurrent}\}. 
\]
Let $U=(\bu_j)=(u_{i,j})$ be a $d\times d$ symmetric matrix, where $\bu_j$ is the $j$-th column composing $U$, and define a set of $d\times d$ positive definite matrices, $\,\calU$, as 
\begin{align*}
\calU &= \Big\{ U=(\bu_j) : \mbox{$U$ is positive definite}\ \mbox{and}\ \langle\ba(A),\bu_j\rangle < 0\ \mbox{for all $A\in\calN_p$ and all $j\in A$} \Big\}, 
\end{align*}
where $\langle\bx,\bx'\rangle$ is the inner product of vectors $\bx,\bx'\in\mathbb{R}^d$ and $\ba(A)$ is the mean increment vector evaluated by the stationary distribution of the induced Markov chain $\calL^A$ being positive recurrent. 
Theorem 4.1 of Ozawa \cite{Ozawa12} can easily be extended to the case where the $d$-dimensional MMRRW, $\calL$, is semi-irreducible and we obtain the following theorem, which gives a condition on which $\calL$ is positive recurrent.

\begin{theorem} \label{th:dpositive}
The $d$-dimensional semi-irreducible MMRRW is positive recurrent in our sense if $\,\calU\ne\emptyset$.
\end{theorem}

Since the extension is straightforward and the theorem can be proved by using Corollary \ref{co:Foster2} in a manner similar to that used for proving the original one in Ozawa \cite{Ozawa12}, we omit the proof of Theorem \ref{th:dpositive}. This theorem asserts that if we can find out a positive definite matrix whose elements satisfy a system of linear inequalities, then the corresponding semi-irreducible MMRRW is positive recurrent. 
We will use Theorem \ref{th:dpositive} to derive a condition on which our two-station network is positive recurrent and hence we do not use Theorem \ref{th:Foster1} and Corollary \ref{co:Foster2} for the purpose.

%%%%%%%%%%%%%%%%%%%%%%%
%
% Section 4
%
%%%%%%%%%%%%%%%%%%%%%%%
%
\section{Stability and instability of the two-station network} \label{sec:positive}

Here we return to the two-station network considered in Section \ref{sec:modelandresults}, where the behavior of the network was represented as the continuous-time Markov chain $\{\bY(t)\}=\{(\bX(t),\bJ(t))\}$, and prove Theorem \ref{th:main0}. 
For the purpose, we construct a discrete-time Markov chain $\{\bY_n\}$, which is a four-dimensional MMRRW, from $\{\bY(t)\}$ by the uniformization technique and apply the theorems in the previous section to the MMRRW $\{\bY_n\}$. 
In this section, we use the same notations as those in Section \ref{sec:modelandresults}.

\subsection{Markov modulated reflecting random walk} \label{sec:modelandinducded}

Remind that the Markov chain $\{\bY(t)\}$ is assumed to be semi-irreducible, its state space is given by $\calS=\mathbb{Z}_+^4\times S_0$ and its infinitesimal generator is given by $Q=(Q(\bx,\bx'),\,\bx,\bx'\in\mathbb{Z}_+^4)$. 
Since the front process of $\{\bY(t)\}$, $\{\bX(t)\}$, is skip-free in all coordinates, the diagonal elements of $Q$ are bounded, i.e., for some positive number $\nu$, we have
\[
\sup_{\bx\in\mathbb{Z}_+^4} \max_{j\in S_0} \left|[Q(\bx,\bx)]_{j,j} \right| \le \nu <\infty, 
\]
where, for a matrix $A$, we denote by $[A]_{i,j}$ the $(i,j)$-element of $A$.
Define a transition probability matrix $P=(P(\bx,\bx'), \bx,\bx'\in\mathbb{Z}_+^4)$ as $P = I + Q/\nu$, 
where we have, for $\bx,\bx'\in\mathbb{Z}_+^4$, 
\[
P(\bx,\bx') = \left\{ \begin{array}{ll} 
I+Q(\bx,\bx')/\nu & \mbox{if $\bx=\bx'$}, \cr 
Q(\bx,\bx')/\nu & \mbox{if $\bx\ne\bx'$}. \end{array} \right.
\]
By setting $\nu$ at an appropriate value, $P$ becomes aperiodic; hence, hereafter, we assume $P$ is aperiodic. 
Let $\{\bY_n\}=\{(\bX_n,\bJ_n)\}$ be a discrete-time Markov chain governed by the transition probability matrix $P$, where $\bX_n=(X_{1,n},X_{2,n},X_{3,n},X_{4,n})$, and $\bJ_n=(J^a_{1,n},J^a_{3,n},J^s_{1,n},J^s_{2,n})$. By the definition of $P$, $\{\bY_n\}$ is semi-irreducible with the same irreducible class as that of $\{\bY(t)\}$. 
Furthermore, $\{\bY(t)\}$ and $\{\bY_n\}$ has the same stationary distribution if it exists. Hence, we see that $\{\bY(t)\}$ is positive recurrent in our sense if and only if $\{\bY_n\}$ is; we also see that $\{\bY(t)\}$ is transient in our sense if and only if $\{\bY_n\}$ is. For our purpose, it is, therefore, sufficient to analyze stability of $\{\bY_n\}$. 
The state space of $\{\bY_n\}$ is given by $\calS=\mathbb{Z}_+^4\times S_0$. For $A\subset N$, We dente by $\calB^A$ a boundary face of $\calS$ given by 
\[
\calB^A = \{(\bx,\bj)\in\mathbb{Z}_+^4\times S_0: \varphi(\bx)=A\},
\]
where $\varphi$ is the function defined in Section \ref{sec:preliminary}. 
Since the transition rates of $\{\bY(t)\}$ are space-homogeneous in each boundary face, the transition probabilities of $\{\bY_n\}$ are also space-homogeneous in each boundary face. Furthermore, the front process of $\{\bY_n\}$, $\{\bX_n\}$, is skip-free in all coordinates. Hence, the Markov chain $\{\bY_n\}$ is a four-dimensional MMRRW. 
Note that, since we have $\varphi(\bx)=S_0$ for all $\bx\in\mathbb{Z}_+^4$, the state space of the background process $\{\bJ_n\}$ is given by $S_0$ for all $\bx\in\mathbb{Z}_+^4$. 

We consider that the Markov chain $\{\bY_n\}$ represents the behavior of a discrete-time two-station network corresponding to the continuous-time two-station network described in Section \ref{sec:modelandresults}. The discrete-time two-station network has the same network configuration as that of the original two-station network. 
For $i=1,3$, the arrival process of Q$_i$ is a discrete-time MAP with representation $(C_i,D_i)$, where $C_i=I+\bar{C}_i/\nu$ and $D_i=\bar{D}_i/\nu$, and, for $i=1,2$, the service process of station $i$ is a discrete-time MSP whose representation is given by $U_i^{k^*l},k,l=0,+$, \ $U_i^{kl^*},k,l=0,+$, and 
\begin{align*}
&T_i^{kl}=I+\bar{T}_i^{kl}/\nu,\ k,l=0,+,\quad
T_i^{k^*l}=\bar{T}_i^{k^*l}/\nu,\ k=1,2,\,l=0,+,\cr
&T_i^{kl^*}=\bar{T}_i^{kl^*}/\nu,\ k=0,+,\,l=1,2.
\end{align*}
For $i=1,3$, the mean arrival rate of Q$_i$ is given by $\lambda_i/\nu$ and, for $i\in N$, the mean service time of class-$i$ customers is given by $\nu h_i$. Hence, for $i\in N$, the offered of Q$_i$ is given by $\rho_i$, which is the same as that of the original two-station network. 
Let $A$ be a non-empty subset of $N$ and denote by $\{\bY_n^A\}=\{(\hat{\bX}_n^{N\setminus A},\hat{\bJ}_n)\}$ the process obtained from $\{\bY_n\}$ by assuming that, for every $i\in A$, Q$_i$ is saturated by customers, where $\hat{\bX}_n^{N\setminus A}=(\hat{X}_{i,n},i\in N\setminus A)$ and $\hat{\bJ}_n=(\hat{J}_{1,n}^a,\hat{J}_{3,n}^a,\hat{J}_{1,n}^s,\hat{J}_{2,n}^s)$; the state space of $\{\bY_n^A\}$ is given by $S^A=\mathbb{Z}_+^{4-|A|}\times S_0$. 
Since $\{\bY^A(t)\}$ is assumed to be semi-irreducible or to have no irreducible class, the Markov chain $\{\bY_n^A\}$ is also semi-irreducible or has no irreducible class; in the latter case, every state of $\{\bY_n^A\}$ is transient. 
If the Markov chain $\{\bY_n^A\}$ is semi-irreducible and positive recurrent, we define, for $i\in N$, the mean input rate and mean output rate of Q$_i$ in the process $\{\bY_n^A\}$, denoted by $\lambda_i^A$ and $\mu_i^A$, as the mean number of customers arriving at Q$_i$ at a time point and that of customers departing from Q$_i$ at a time point, respectively;  $\lambda_i^A$ and $\mu_i^A$ are given by $\lambda_i^A=\bar{\lambda}_i^A/\nu$ and $\mu_i^A=\bar{\mu}_i^A/\nu$. 
We see from the definition of $\{\bY_n^A\}$ that the Markov chain $\{\bY_n^A\}$ is also the induced Markov chain $\calL^A$ of the MMRRW $\{\bY_n\}$, and hence we obtain, for $i\in N$, 
\[
\lambda_i^A-\mu_i^A = a_i(A),
\]
where $a_i(A)$ is the $i$-th element of the mean increment vector $\ba(A)$ of the induced Markov chain $\calL^A$.
For $i\in N$, $\Delta q_i^A$ is, therefore, given in terms of $a_i(A)$ as $\Delta q_i^A = \nu a_i(A)$. Hence, conditions (\ref{eq:dqN}) to (\ref{eq:dq23}) in Section \ref{sec:modelandresults} are equivalent to 
\begin{align}
& a_1(N)>0,\quad a_2(N)<0,\quad a_3(N)>0,\quad a_4(N)<0, \label{eq:aN} \\
& a_1(\{1,2,3\})<0,\quad a_2(\{1,2,3\})>0,\quad a_3(\{1,2,3\})>0, \label{eq:a123} \\
& a_1(\{1,3,4\})>0,\quad a_3(\{1,3,4\})<0,\quad a_4(\{1,3,4\})>0, \label{eq:a134} \\
& a_1(\{1,4\})>0,\quad a_4(\{1,4\})<0, \label{eq:a14} \\
& a_2(\{2,3\})<0,\quad a_3(\{2,3\})>0,  \label{eq:a23}
\end{align}
where we have
\begin{align*}
&\ba(N)=(a_1(N),a_2(N),a_3(N),a_4(N)), \cr 
&\ba(\{1,2,3\})=(a_1(\{1,2,3\}),a_2(\{1,2,3\}),a_3(\{1,2,3\}),0), \cr
&\ba(\{1,3,4\})=(a_1(\{1,3,4\}),0,a_3(\{1,3,4\}),a_4(\{1,3,4\})), \cr
&\ba(\{1,4\})=(a_1(\{1,4\}),0,0,a_4(\{1,4\})), \cr
&\ba(\{2,3\})=(0,a_2(\{2,3\}),a_3(\{2,3\}),0).
\end{align*}
We present the following proposition, which makes clear which induced Markov chain is positive recurrent. 

\begin{proposition} \label{pr:inducedMarkovchain}
Induce Markov chain $\calL^N$ is always positive recurrent. 
Under condition (\ref{eq:aN}), if induced Markov chains $\calL^{\{1,2,3\}}$ and $\calL^{\{1,3,4\}}$ are semi-irreducible, then they are positive recurrent; on the other hand, induced Markov chains $\calL^{\{1,2,4\}}$ and $\calL^{\{2,3,4\}}$ are transient. 
Under conditions (\ref{eq:aN}) to (\ref{eq:a134}), if induced Markov chains $\calL^{\{1,4\}}$ and $\calL^{\{2,3\}}$ are semi-irreducible, then they are positive recurrent; on the other hand, induced Markov chains $\calL^{\{1,2\}}$, $\calL^{\{1,3\}}$, $\calL^{\{2,4\}}$ and $\calL^{\{3,4\}}$ are transient. 
Under conditions (\ref{eq:aN}) to (\ref{eq:a23}), induced Markov chains $\calL^{\{1\}}$, $\calL^{\{2\}}$, $\calL^{\{3\}}$ and $\calL^{\{4\}}$ are transient. 
\end{proposition}
Note that, in this proposition, every induced Markov chain of $\{\bY_n\}$ is assumed to be semi-irreducible or to have no irreducible class. 
\begin{proof}[Proof of Proposition \ref{pr:inducedMarkovchain}]
First we must say that the results of Ozawa \cite{Ozawa12} can be applied to MMRRWs that are semi-irreducible. 
Since every induced Markov chain of $\{\bY_n\}$ is assumed to be semi-irreducible or to have no irreducible class and induce Markov chain $\calL^N$ is a finite Markov chain, $\calL^N$ must be semi-irreducible and positive recurrent; hence the mean increment vector $\ba(N)$ exists. 

Induced Markov chain $\calL^{\{1,2,3\}}$ is a one-dimensional MMRRW, denoted by $\{(\hat{X}_{4,n},\hat{\bJ}_n)\}$, and it has the induced Markov chain, which we denote by $\tilde{\calL}^{\{4\}}$; $\tilde{\calL}^{\{4\}}$ is equivalent to the induced Markov chain $\calL^N$ and its mean increment vector, denoted by $\tilde{\ba}(\{4\})$, is given by $\tilde{\ba}(\{4\})=(a_4(N))$, where $a_4(N)$ is the fourth element of the mean increment vector $\ba(N)$ of $\calL^N$. 
Since $\calL^{\{1,2,3\}}$ is semi-irreducible and $a_4(N)<0$, we see from the results of Ozawa \cite{Ozawa12} that it is positive recurrent and the mean increment vector $\ba(\{1,2,3\})$ exists. 
Analogously, we see that $\calL^{\{1,3,4\}}$ is also positive recurrent and the mean increment vector $\ba(\{1,3,4\})$ exists. 
Induced Markov chain $\calL^{\{1,2,4\}}$ is a one-dimensional MMRRW, denoted by $\{(\hat{X}_{3,n},\hat{\bJ}_n)\}$, and it has the induced Markov chain $\tilde{\calL}^{\{3\}}$, which is positive recurrent and whose mean increment vector is given by $\tilde{\ba}(\{3\})=(a_3(N))$. 
If $\calL^{\{1,2,4\}}$ is semi-irreducible, it is transient since we have $a_3(N)>0$; if $\calL^{\{1,2,4\}}$ has no irreducible class, every state of $\calL^{\{1,2,4\}}$ must be transient. 
Analogously, we see that $\calL^{\{2,3,4\}}$ is transient. 

Induced Markov chain $\calL^{\{1,2\}}$ is a two-dimensional MMRRW, denoted by $\{((\hat{X}_{3,n},\hat{X}_{4,n}),\hat{\bJ}_n)\}$, and it has three induced Markov chains, which we denote by $\tilde{\calL}^{\{3,4\}}$, $\tilde{\calL}^{\{3\}}$ and $\tilde{\calL}^{\{4\}}$, respectively. 
$\tilde{\calL}^{\{3,4\}}$ is equivalent to $\calL^N$ and its mean increment vector, denoted by $\tilde{\ba}(\{3,4\})$, is given by $\tilde{\ba}(\{3,4\})=(a_3(N),a_4(N))$; $\tilde{\calL}^{\{3\}}$ is equivalent to $\calL^{\{1,2,3\}}$ and its mean increment vector $\tilde{\ba}(\{3\})$ is given by $\tilde{\ba}(\{3\})=(a_3(\{1,2,3\}),0)$; $\tilde{\calL}^{\{4\}}$ is equivalent to $\calL^{\{1,2,4\}}$ and it is transient. 
By Theorem 5.1 of Ozawa \cite{Ozawa12}, if $\calL^{\{1,2\}}$ is semi-irreducible, it is transient since we have that $a_3(N)>0$, $a_4(N)<0$ and $a_3(\{1,2,3\})>0$; if $\calL^{\{1,2\}}$ has no irreducible class, every state of $\calL^{\{1,2\}}$ must be transient. 
Analogously, we see that $\calL^{\{3,4\}}$ is transient. 
Induced Markov chain $\calL^{\{1,3\}}$ is a two-dimensional MMRRW, denoted by $\{((\hat{X}_{2,n},\hat{X}_{4,n}),\hat{\bJ}_n)\}$, and it has three induced Markov chains, denoted by $\tilde{\calL}^{\{2,4\}}$, $\tilde{\calL}^{\{2\}}$ and $\tilde{\calL}^{\{4\}}$, respectively. 
$\tilde{\calL}^{\{2,4\}}$ is equivalent to $\calL^N$ and its mean increment vector, denoted by $\tilde{\ba}(\{2,4\})$, is given by $\tilde{\ba}(\{2,4\})=(a_2(N),a_4(N))$; $\tilde{\calL}^{\{2\}}$ is equivalent to $\calL^{\{1,2,3\}}$ and its mean increment vector $\tilde{\ba}(\{2\})$ is given by $\tilde{\ba}(\{2\})=(a_2(\{1,2,3\}),0)$; $\tilde{\calL}^{\{4\}}$ is equivalent to $\calL^{\{1,3,4\}}$ and its mean increment vector $\tilde{\ba}(\{4\})$ is given by $\tilde{\ba}(\{4\})=(0,a_4(\{1,3,4\}))$. 
By Theorem 5.1 of Ozawa \cite{Ozawa12}, if $\calL^{\{1,3\}}$ is semi-irreducible, it is transient since we have that $a_2(N)<0$, $a_4(N)<0$, $a_2(\{1,2,3\})>0$ and $a_4(\{1,3,4\})>0$; if $\calL^{\{1,3\}}$ has no irreducible class, every state of $\calL^{\{1,3\}}$ must be transient. 
Induced Markov chain $\calL^{\{1,4\}}$ is a two-dimensional MMRRW, denoted by $\{((\hat{X}_{2,n},\hat{X}_{3,n}),\hat{\bJ}_n)\}$, and it has three induced Markov chains, denoted by $\tilde{\calL}^{\{2,3\}}$, $\tilde{\calL}^{\{2\}}$ and $\tilde{\calL}^{\{3\}}$, respectively. 
$\tilde{\calL}^{\{2,3\}}$ is equivalent to $\calL^N$ and its mean increment vector, denoted by $\tilde{\ba}(\{2,3\})$, is given by $\tilde{\ba}(\{2,3\})=(a_2(N),a_3(N))$; $\tilde{\calL}^{\{2\}}$ is equivalent to $\calL^{\{1,2,4\}}$ and it is transient; $\tilde{\calL}^{\{3\}}$ is equivalent to $\calL^{\{1,3,4\}}$ and its mean increment vector $\tilde{\ba}(\{3\})$ is given by $\tilde{\ba}(\{3\})=(0,a_3(\{1,3,4\}))$. 
By Theorem 5.1 of Ozawa \cite{Ozawa12}, since $\calL^{\{1,4\}}$ is semi-irreducible and we have that $a_2(N)<0$, $a_3(N)>0$ and $a_3(\{1,3,4\}<0$, $\calL^{\{1,4\}}$ is positive recurrent and the mean increment vector $\ba(\{1,4\})$ exists.  
Analogously, we see that $\calL^{\{2,3\}}$ is positive recurrent and the mean increment vector $\ba(\{2,3\})$ exists. 
Induced Markov chain $\calL^{\{2,4\}}$ is a two-dimensional MMRRW, denoted by $\{((\hat{X}_{1,n},\hat{X}_{3,n}),\hat{\bJ}_n)\}$, and it has three induced Markov chains, denoted by $\tilde{\calL}^{\{1,3\}}$, $\tilde{\calL}^{\{1\}}$ and $\tilde{\calL}^{\{3\}}$, respectively. 
$\tilde{\calL}^{\{1,3\}}$ is equivalent to $\calL^N$ and its mean increment vector, denoted by $\tilde{\ba}(\{1,3\})$, is given by $\tilde{\ba}(\{1,3\})=(a_1(N),a_3(N))$, where $a_1(N)>0$, $a_3(N)>0$. Hence, by Theorem 5.1 of Ozawa \cite{Ozawa12}, if $\calL^{\{1,3\}}$ is semi-irreducible, it is transient; if $\calL^{\{1,3\}}$ has no irreducible class, every state of $\calL^{\{1,3\}}$ must be transient. 

Induced Markov chain $\calL^{\{1\}}$ is a three-dimensional MMRRW, denoted by $\{((\hat{X}_{2,n},\hat{X}_{3,n},\hat{X}_{4,n}),\hat{\bJ}_n)\}$, and it has seven induced Markov chains, which we denote by $\tilde{\calL}^{\{2,3,4\}}$, $\tilde{\calL}^{\{2,3\}}$, $\tilde{\calL}^{\{2,4\}}$, $\tilde{\calL}^{\{3,4\}}$, $\tilde{\calL}^{\{2\}}$, $\tilde{\calL}^{\{3\}}$ and $\tilde{\calL}^{\{4\}}$, respectively. 
$\tilde{\calL}^{\{2,3,4\}}$ is equivalent to $\calL^N$ and its mean increment vector, denoted by $\tilde{\ba}(\{2,3,4\})$, is given by $\tilde{\ba}(\{2,3,4\})=(a_2(N),a_3(N),a_4(N))$; $\tilde{\calL}^{\{2,3\}}$ is equivalent to $\calL^{\{1,2,3\}}$ and its mean increment vector $\tilde{\ba}(\{2,3\})$ is given by $\tilde{\ba}(\{2,3\})=(a_2(\{1,2,3\}),a_3(\{1,2,3\}),0)$; other induced Markov chains of $\calL^{\{1\}}$ do not affect the stability of $\calL^{\{1\}}$.  
Since we have that $a_2(N)<0$, $a_3(N)>0$, $a_4(N)<0$, $a_2(\{1,2,3\})>0$ and $a_3(\{1,2,3\})>0$, the three-dimensional MMRRW $\calL^{\{1\}}$ belongs C2-3-1 in Table 2 of Ozawa \cite{Ozawa12}. Hence, if $\calL^{\{1\}}$ is semi-irreducible, it is transient; if $\calL^{\{1\}}$ has no irreducible class, every state of $\calL^{\{1\}}$ must be transient. 
Analogously, we see that $\calL^{\{3\}}$ is transient. 
Induced Markov chain $\calL^{\{2\}}$ is a three-dimensional MMRRW, denoted by $\{((\hat{X}_{1,n},\hat{X}_{3,n},\hat{X}_{4,n}),\hat{\bJ}_n)\}$, and it has seven induced Markov chains, denoted by $\tilde{\calL}^{\{1,3,4\}}$, $\tilde{\calL}^{\{1,3\}}$, $\tilde{\calL}^{\{1,4\}}$, $\tilde{\calL}^{\{3,4\}}$, $\tilde{\calL}^{\{1\}}$, $\tilde{\calL}^{\{3\}}$ and $\tilde{\calL}^{\{4\}}$, respectively. 
$\tilde{\calL}^{\{1,3,4\}}$ is equivalent to $\calL^N$ and its mean increment vector, denoted by $\tilde{\ba}(\{1,3,4\})$, is given by $\tilde{\ba}(\{1,3,4\})=(a_1(N),a_3(N),a_4(N))$; $\tilde{\calL}^{\{1,3\}}$ is equivalent to $\calL^{\{1,2,3\}}$ and its mean increment vector $\tilde{\ba}(\{1,3\})$ is given by $\tilde{\ba}(\{1,3\})=(a_1(\{1,2,3\}),a_3(\{1,2,3\}),0)$; $\tilde{\calL}^{\{3\}}$ is equivalent to $\calL^{\{2,3\}}$ and its mean increment vector $\tilde{\ba}(\{3\})$ is given by $\tilde{\ba}(\{3\})=(0,a_3(\{2,3\}),0)$; other induced Markov chains of $\calL^{\{2\}}$ are transient and they do not affect the stability of $\calL^{\{2\}}$.  
Since we have that $a_1(N)>0$, $a_3(N)>0$, $a_4(N)<0$, $a_1(\{1,2,3\})<0$, $a_3(\{1,2,3\})>0$ and $a_3(\{2,3\})>0$, the three-dimensional MMRRW $\calL^{\{2\}}$ belongs C3-2-2 in Table 3 of Ozawa \cite{Ozawa12}. Hence, if $\calL^{\{2\}}$ is semi-irreducible, it is transient; if $\calL^{\{2\}}$ has no irreducible class, every state of $\calL^{\{2\}}$ must be transient. 
Analogously, we see that $\calL^{\{4\}}$ is transient and this completes the proof. 
\end{proof}

We denote by $\calN_p$ the index set of the induced Markov chains of $\{\bY_n\}$ that are semi-irreducible and positive recurrent; $\calN_p$ is given by 
\begin{equation}
\calN_p =\{N, \{1,2,3\}, \{1,3,4\}, \{1,4\}, \{2,3\}\}, 
\label{eq:Np}
\end{equation}
which has already been appeared in Section \ref{sec:modelandresults}.

%%%%%%%%%%%%%%%%%%%%%%%%%%%%%%%%%%%%%%%%%%%%%%%%%%%%%%%%%%%%%%%%%%%
\subsection{Proof of Theorem \ref{th:main0}}

For $A\in\calN_p$ and for $i,j\in A$, let $r^A_{i,j}$ be defined as $r^A_{i,j} = \left| a_i(A)/a_j(A) \right|$. As mentioned in the previous subsection, for $A\in\calN_p$ and for $i\in N$, equality $\Delta q_i^A=\nu a_i(A)$ holds; hence $r_1$ and $r_2$ in Theorem \ref{th:main0} are given by 
\begin{equation}
r_1 = r^{\{1,2,3\}}_{2,1}\,r^{\{2,3\}}_{3,2}+r^{\{1,2,3\}}_{3,1},\quad 
r_2 = r^{\{1,3,4\}}_{4,3}\,r^{\{1,4\}}_{1,4}+r^{\{1,3,4\}}_{1,3}. 
\end{equation}
We prove the following corollary under conditions (\ref{eq:aN}) to (\ref{eq:a23}), instead of directly proving Theorem \ref{th:main0}. 

\begin{corollary} \label{co:stability} 
Assume that $r_{1,4}^N \le r_{1,4}^{\{1,4\}}$ and $r_{3,2}^N < r_{3,2}^{\{2,3\}}$ or that $r_{1,4}^N < r_{1,4}^{\{1,4\}}$ and $r_{3,2}^N \le r_{3,2}^{\{2,3\}}$. 
Then, the semi-irreducible Markov chain $\{\bY_n\}$ is positive recurrent in our sense if $r_1 r_2<1$ and it is transient if $r_1 r_2>1$. 
\end{corollary}

Note that the former condition in Corollary \ref{co:stability} corresponds to condition (\ref{eq:main0cond1}) in Theorem \ref{th:main0} and the latter to condition (\ref{eq:main0cond2}).

%\input{fig2.tex}
%%%%%%%%%%%%%%%%%%%%%%%%%%%%%%%%%%%%%%%%%%%%%%%%%%%%%%%%%%%
%
% fig2.tex
%
%%%%%%%%%%%%%%%%%%%%%%%%%%%%%%%%%%%%%%%%%%%%%%%%%%%%%%%%%%%
\begin{figure}[bht]
\begin{center}
\setlength{\unitlength}{1.2mm}
\begin{picture}(70,35)(0,0)
\thinlines
\put(32,35){\makebox(0,0){\normalsize $x_1$}}
\put(40,3){\line(-1,4){7.5}}
\put(13.5,20){\makebox(0,0){\normalsize $x_2$}}
\put(40,3){\line(-3,2){24}}
\put(60,12){\makebox(0,0){\normalsize $x_3$}}
\put(40,3){\line(2,1){17}}
\put(58,28.5){\makebox(0,0){\normalsize $x_4$}}
\put(40,3){\line(2,3){16}}
\put(40,0){\makebox(0,0){\normalsize $O$}}
\thicklines
\put(36.5,29){\makebox(0,0){\normalsize $P_1$}}
\put(33.8,28){\vector(-1,-4){3.5}}
\multiput(20,16)(2,-0.5){5}{\line(1,0){1}}
\put(31,11.5){\makebox(0,0){\normalsize $P_2$}}
\put(30.5,13.8){\vector(4,-1){20.5}}
\put(52,6){\makebox(0,0){\normalsize $P_3$}}
\put(51,8.5){\vector(-1,3){4.5}}
\multiput(51,21)(-2,0.5){3}{\line(1,0){1}}
\put(47,25){\makebox(0,0){\normalsize $P_4$}}
\put(47,22){\vector(-4,1){12.5}}
\put(38.5,21.5){\makebox(0,0){\normalsize $P_5$}}
\end{picture}
\caption{A spiral path on the second vector field.}
\label{fig:spiral}
\end{center}
\end{figure}
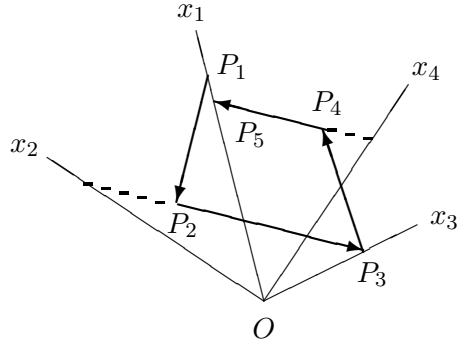
%%%%%%%%%%%%%%%%%%%%%%%%%%%%%%%%%%%%%%%%%%%%%%%%%%%%%%%%%%%

\begin{remark} \label{re:vectorfield}
Inequality $r_1 r_2<1$ corresponds to the condition that spiral paths on the second vector field introduced by Malyshev \cite{Malyshev93} reach the origin. Here we explain this point. 

For $A\subset N$, let $\bar{\calB}^A$ be a boundary of orthant $\mathbb{R}_+^4$, defined as 
\[
\bar{\calB}^A = \{(x_1,x_2,x_3,x_4)\in\mathbb{R}_+^4 : \mbox{$x_i>0$ for $i\in A$; $x_i=0$ for $i\in N\setminus A$} \}, 
\]
where $\bar{\calB}^\emptyset$ includes only the origin and $\bar{\calB}^N$ is the interior of $\mathbb{R}_+^4$. 
The second vector field generated from $\{\bY_n\}$ is constructed by assigning the mean increment vector $\ba(A)$ to each $\bx\in\bar{\calB}^A$ if the induced Markov chain $\calL^A$ is positive recurrent, where $A$ is a non-empty subset of $N$; if $\calL^A$ is transient, we consider the vector assigned to $\bx\in\bar{\calB}^A$ is undetermined. 
Let us consider a spiral path on the second vector field, depicted in Fig.~\ref{fig:spiral}. It starts from point $P_1$ on the $x_1$-axis; while $P_1$ is a point on boundary $\bar{\calB}^{\{1\}}$ and the vector on $P_1$ is undetermined since $\calL^{\{1\}}$ is transient, we suppose the path continues toward boundary $\bar{\calB}^{\{1,2,3\}}$. 
Since vector $\ba(\{1,2,3\})$ is assigned to each point on $\bar{\calB}^{\{1,2,3\}}$, the path continues further through $\bar{\calB}^{\{1,2,3\}}$ toward boundary $\bar{\calB}^{\{2,3\}}$, following the vector $\ba(\{1,2,3\})$, and reaches pint $P_2$ on $\bar{\calB}^{\{2,3\}}$. 
Then, following vector $\ba(\{2,3\})$, the path continues through $\bar{\calB}^{\{2,3\}}$ toward boundary $\bar{\calB}^{\{3\}}$ and reaches pint $P_3$ on $\bar{\calB}^{\{3\}}$. 
While the vector on $P_3$ is undetermined since $\calL^{\{3\}}$ is transient, we also suppose the path continues toward boundary $\bar{\calB}^{\{1,3,4\}}$. The path continues further through $\bar{\calB}^{\{1,3,4\}}$ toward boundary $\bar{\calB}^{\{1,4\}}$, following vector $\ba(\{1,3,4\})$, and reaches pint $P_4$ on $\bar{\calB}^{\{1,4\}}$. 
Then, following vector $\ba(\{1,4\})$, the path continues through $\bar{\calB}^{\{1,4\}}$ toward $\bar{\calB}^{\{1\}}$ and reaches pint $P_5$ on $\bar{\calB}^{\{1\}}$. 
After that, the path continues in the same way. 
We, therefore, see that if $P_5$ is closer to the origin than $P_1$, the spiral path eventually reaches the origin. Supposing $P_1=(1,0,0,0)$, we obtain, through some calculation, that $P_3=(0,0,x_3,0)$, where $x_3$ is given by
\[
x_3 = \frac{a_2(\{1,2,3\})}{a_1(\{1,2,3\})} \frac{a_3(\{2,3\})}{a_2(\{2,3\})}-\frac{a_3(\{1,2,3\})}{a_1(\{1,2,3\})} = r_1; 
\]
Supporting $P_3=(0,0,1,0)$, we also obtain $P_5=(x_1,0,0,0)$, where $x_1$ is given by
\[
x_1 = \frac{a_4(\{1,3,4\})}{a_3(\{1,3,4\})} \frac{a_1(\{1,4\})}{a_4(\{1,4\})}-\frac{a_1(\{1,3,4\})}{a_3(\{1,3,4\})} = r_2. 
\]
Hence, we see that if $r_1 r_2 <1$, the spiral path eventually reaches the origin. 
Here we should note that this is just intuitive derivation of the condition $r_1 r_2 <1$ and it is not a proof of Corollary \ref{co:stability}. 
\end{remark}

%%%%%%%%%%%%%%%%%%%%%%%%%%%%%%%%%
We will prove Corollary \ref{co:stability} separating it into a positive-recurrence part and transience part. Before doing it, we give the following proposition.
\begin{proposition} \label{pr:pdefiniteU}
For $i\in N=\{1,2,3,4\}$, let $u_{ii}$ be a positive real number and let $c$ be defined as 
\[
c=u_{11}\,u_{22}\,u_{33}\,u_{44}.
\]
For $\varepsilon$ such that $0\le\varepsilon\le c$, let $U(\varepsilon)=(u_{ij}(\varepsilon))$ be a 4-dimensional square matrix defined by 
\[
u_{i,j}(\varepsilon) = \left\{ \begin{array}{ll}
u_{ii}, & i=j, \cr
\sqrt{u_{ii}u_{jj}\left(1-\frac{\varepsilon}{c}\right)}, & i\ne j. \end{array}\right.
\]
Then $U(\varepsilon)$ is symmetric and there exists a positive number $\varepsilon_0$ such that, for every $\varepsilon\in(0,\varepsilon_0)$, $U(\varepsilon)$ is positive definite.
\end{proposition}
\begin{proof}
For $k\in\{1,2,3\}$, let $U_k(\varepsilon)$ be the $k$-th leading principal submatrix of $U(\varepsilon)$, i.e., $U_k(\varepsilon)=(u_{i,j}(\varepsilon),1\le i,j\le k)$. Through some calculation, we obtain 
\begin{align*}
&|U_1(\varepsilon)| = u_{11}, \cr
&|U_2(\varepsilon)| = u_{11}u_{22}\,\frac{\varepsilon}{c}, \cr
&|U_3(\varepsilon)| = \frac{1}{u_{44}}\left(3\varepsilon+2 c \left( \left(1-\frac{\varepsilon}{c}\right)^{\frac{3}{2}}-1\right) \right),\cr
%\frac{d}{d \varepsilon}|U_3(\varepsilon)|\Big|_{\varepsilon=0}=0, 
&|U(\varepsilon)| = \frac{1}{c}\left(3\varepsilon^2+2 c^2\left(-3\left(1-\frac{\varepsilon}{c}\right)^2+4 \left(1-\frac{\varepsilon}{c}\right)^{\frac{3}{2}}-1\right)\right).
\end{align*}
From these formulae, we see that all the leading principal minors of $U(\varepsilon)$ including the determinant of $U(\varepsilon)$ are positive for sufficiently small any positive $\varepsilon$. Thus, by Sylvester's criterion (see, for example, Horn and Johnson \cite{Horn13}), $U(\varepsilon)$ is positive definite for such $\varepsilon$. 
\end{proof}

\medskip
%%%%%%%%%%%%%%%%%
\begin{proof}[Proof of Corollary \ref{co:stability} (positive recurrence)]
Let $U(\varepsilon)=(u_{ij}(\varepsilon))=(\bu_1(\varepsilon),\bu_2(\varepsilon),\bu_3(\varepsilon),\bu_4(\varepsilon))$ be the positive-definite 4-dimensional square matrix given in Proposition \ref{pr:pdefiniteU}, where, for $j\in N$, $\bu_j(\varepsilon)$ is the $j$-th column of $U(\varepsilon)$. 
In order to prove $\{\bY_n\}$ being positive recurrent, it suffices from Theorem \ref{th:dpositive} to demonstrate that there exist positive numbers $u_{11}$, $u_{22}$, $u_{33}$ and $u_{44}$ such that, for a sufficiently small $\varepsilon$, every $A\in\calN_p$ and every $k\in A$, we have $\langle\ba(A),\bu_k(\varepsilon)\rangle<0$. 
To this end, we set $u_{11}$ and $u_{22}$ so that they satisfy 
\begin{align}
&r_{3,2}^{\{2,3\}} \sqrt{u_{11}} - r_1 \sqrt{u_{22}}>0, \label{eq:u11} \\
&\sqrt{u_{22}}-r_2\,r_{3,2}^{\{2,3\}} \sqrt{u_{11}} >0; \label{eq:u22}
\end{align}
it is possible because of the assumption $r_1 r_2 <1$. Furthermore, we assume $u_{33}$ and $u_{44}$ to be given in terms of $u_{11}$ and $u_{22}$ by 
\begin{align}
&u_{33} = (r_{3,2}^{\{2,3\}})^2(u_{22}-\delta),\\ 
&u_{44} = (r_{1,4}^{\{1,4\}})^2(u_{11}+\delta), 
\end{align}
where $\delta$ is a small positive number and we assume $\varepsilon$ to be sufficiently smaller than $\delta$.

First we consider the case of $A=\{1,4\}$. In this case, we have 
\[
u_{14}(\varepsilon)^2-(r_{1,4}^{\{1,4\}})^2 u_{11}^2 = (r_{1,4}^{\{1,4\}})^2 u_{11} \left(\delta-\varepsilon(u_{11}+\delta)/c\right)>0,
\]
where we use the fact that $\varepsilon$ is sufficiently smaller than $\delta$. Since $a_4(A)<0$, this implies 
\[
\langle\ba(\{1,4\}),\bu_1(\varepsilon)\rangle = a_4(\{1,4\}) (-r_{1,4}^{\{1,4\}} u_{11}+u_{14}(\varepsilon)) < 0.
\]
Analogously, we obtain $\langle\ba(\{1,4\}),\bu_4(\varepsilon)\rangle<0$. 
In the case of $A=\{2,3\}$, we also obtain $\langle\ba(\{2,3\}),\bu_2(\varepsilon)\rangle<0$ and $\langle\ba(\{2,3\}),\bu_3(\varepsilon)\rangle<0$ in a way similar to that used in the case of $A=\{1,4\}$.

Next we consider the case of $A=\{1,2,3\}$. In this case, we have 
\begin{align*}
\langle\ba(\{1,2,3\}),\bu_1(\varepsilon)\rangle 
&= a_1(\{1,2,3\}) \left( u_{11}-r_{2,1}^{\{1,2,3\}} u_{12}(\varepsilon) - r_{3,1}^{\{1,2,3\}} u_{13}(\varepsilon) \right) \cr
&= a_1(\{1,2,3\}) \biggl(u_{11}-r_{2,1}^{\{1,2,3\}} \sqrt{u_{11}u_{22}\left(1-\frac{\varepsilon}{c}\right)} \cr
&\qquad\qquad\qquad - r_{3,1}^{\{1,2,3\}}  \sqrt{u_{11} \left(r_{3,2}^{\{2,3\}}\right)^{-2} (u_{22}-\delta) \left(1-\frac{\varepsilon}{c} \right)} \biggr) \cr
&= a_1(\{1,2,3\}) \sqrt{u_{11}} \left(r_{3,2}^{\{2,3\}}\right)^{-1} \left( r_{3,2}^{\{2,3\}} \sqrt{u_{11}}-r_1\sqrt{u_{22}} \right) + \mathcal{O}(\varepsilon+\delta), 
\end{align*}
where we say $f(x)=\mathcal{O}(g(x))$ as $x\to a$ if there exist positive constants $M$ and $m$ such that, for any $x\in(a-m,a+m)\setminus\{a\}$, $|f(x)|\le M |g(x)|$. Since both $\varepsilon$ and $\delta$ are sufficiently small and we have $a_1(\{1,2,3\})<0$ and inequality (\ref{eq:u11}), the expression above implies $\langle\ba(\{1,2,3\}),\bu_1(\varepsilon)\rangle<0$. 
Analogously, we obtain $\langle\ba(\{1,2,3\}),\bu_2(\varepsilon)\rangle<0$ and $\langle\ba(\{1,2,3\}),\bu_3(\varepsilon)\rangle<0$. 
In the case of $A=\{1,3,4\}$, we also obtain, for $k\in \{1,3,4\}$, $\langle\ba(\{1,3,4\}),\bu_k(\varepsilon)\rangle<0$ in a way similar to that used in the case of $A=\{1,2,3\}$, where we use inequality (\ref{eq:u22}) instead of (\ref{eq:u11}).

Lastly, we consider the case of $A=N$. In this case, we have 
\begin{align*}
\langle\ba(N),\bu_1(\varepsilon)\rangle 
&= a_4(N) \left( r_{1,4}^{\{1,4\}}-r_{1,4}^N \right) u_{11} + a_2(N) \left(r_{3,2}^{\{2,3\}}\right)^{-1} \left( r_{3,2}^{\{2,3\}}-r_{3,2}^N \right) \sqrt{u_{11}u_{22}} \cr
&\qquad + \mathcal{O}(\varepsilon+\delta), 
\end{align*}
and, by the assumption of Corollary \ref{co:stability}, this implies $\langle\ba(N),\bu_1(\varepsilon)\rangle<0$. 
Analogously, we obtain, for $k\in\{2,3,4\}$, $\langle\ba(N),\bu_k(\varepsilon)\rangle<0$, and this completes the proof. 
\end{proof}

\medskip
%%%%%%%%%%%%%%%%%%%%
\begin{proof}[Proof of Corollary \ref{co:stability} (transience)] 
Transience of $\{\bY_n\}$ can be proved in a way similar to that used in the proof of Theorem 5.2 of Ozawa \cite{Ozawa12}. Hence we will explain only outline. 

%%%%%%%%%%%%%%%%%%%
In order to apply Theorem \ref{th:Markov_transient}, first we construct an embedded Markov chain of $\{\bY_n\}$, according to Ozawa \cite{Ozawa12}. 
For $A\subset N$ such that $A\ne\emptyset$, let $K_A$ be a positive integer and assume that, for $A, B\subset N$, if $|A|>|B|$ then $K_A < K_B$. We inductively define, for $A\subset N$, a subset $\calV_A\subset\calS$ as follows. 
\begin{description}
\item[Step 1] Set $k := 4$ and $\calV_\emptyset := \calS$.
\item[Step 2] For $A\subset N$ such that $|A|=k$, set $\calV_A := \calV_\emptyset \cap \{ (\bx,\bj)\in\calS : x_l\ge K_A,\ l\in A \}$. 
Set $\calV_\emptyset := \calV_\emptyset\setminus \left(\bigcup_{A\subset N,\,|A|=k} \calV_A \right) $ and $k := k-1$.
\item[Step 3] If $k>0$ then go to Step 2, otherwise we obtain $\calV_A$ for all $A\subset N$.
\end{description}

It is obvious that $\bigcup_{A\subset N} \calV_A = \calS$ and that $\calV_A\cap\calV_B=\emptyset$ for all $A, B\subset N$ such that $A\ne B$. Thus, $\{\calV_A : A\subset N\}$ is a partition of $\calS$. We also see that $\calV_\emptyset$ is finite. 
For $\by\in\calS$, define a function $\tau(\by)$ as 
\[
\tau(\by) = T_A\ \mbox{if $\by\in\calV_A$ for some $A\subset N$}, 
\]
where $T_A,\ A\subset N$, are positive integers. We assume that $T_\emptyset=1$ and that, for $A\subset N$ such that $A\ne\emptyset$, $T_A<K_A$. 
Furthermore, define a time sequence $\{t_n\}$ as $t_{n+1}=t_n+\tau(\bY_{t_n}),\, t_0=0$. The desired embedded Markov chain, denoted by $\{\tilde{\bY}_n\}=\{(\tilde{\bX}_n,\tilde{\bJ}_n)\}$, is defined as $\tilde{\bY}_n = \bY_{t_n},\,n\ge 0$. 

%%%%%%%%%%%%%%%%%%%%
Next we will define a test function. Let $\delta_2$, $\delta_3$ and $\delta_4$ be sufficiently small positive numbers satisfying 
\begin{align}
&\delta_2 > r_{3,2}^{\{2,3\}} \delta_3,\quad 
\delta_3 > r_{4,3}^{\{1,3,4\}} \delta_4, 
\end{align}
and let $c_2$, $c_3$ and $c_4$ be defined as 
\begin{align}
&\frac{1}{c_2} = r_2\,r_{3,2}^{\{2,3\}} - \delta_2,\quad 
\frac{1}{c_3} = r_2 - \delta_3,\quad 
\frac{1}{c_4} = r_{1,4}^{\{1,4\}}-\delta_4. 
\end{align}
Let $c_0$ and $c_1$ be sufficiently large positive numbers and let vectors $\bw_{\{1,2,3\}}$ and $\bw_{\{1,3,4\}}$ be defined as 
\begin{align}
&\bw_{\{1,2,3\}} = \frac{1}{c_1}\left( 1,\,\frac{1}{c_2},\,\frac{1}{c_3},\,\frac{1}{c_4}-c_0 \right),\quad 
\bw_{\{1,3,4\}} = \frac{1}{c_1}\left( 1,\,\frac{1}{c_2}-c_0,\,\frac{1}{c_3},\,\frac{1}{c_4}\right).
\end{align}
Furthermore, let functions $f_{\{1,2,3\}}$ and $f_{\{1,3,4\}}$ be defined as, for $\bx\in\mathbb{R}_+^4$, 
\[
f_{\{1,2,3\}}(\bx) = \langle\bx,\bw_{\{1,2,3\}}\rangle, \quad 
f_{\{1,3,4\}}(\bx) = \langle\bx,\bw_{\{1,3,4\}}\rangle.
\]
Note that both $f_{\{1,2,3\}}(\bx)$ and $f_{\{1,3,4\}}(\bx)$ are linear functions with respect to $\bx$. 
Through some calculation, we obtain 
\begin{align}
f_{\{1,2,3\}}(\ba(\{1,2,3\}) 
&= \frac{a_1(\{1,2,3\})}{c_1} \left( 1-\frac{1}{c_2} r_{2,1}^{\{1,2,3\}}-\frac{1}{c_3} r_{3,1}^{\{1,2,3\}} \right) \cr
&= -\frac{a_1(\{1,2,3\})}{c_1} \left( (r_1 r_2-1)-r_{2,1}^{\{1,2,3\}}\delta_2-r_2\delta_3 \right)>0, 
\label{eq:f123a123}
\end{align}
where we use the facts that $a_1(\{1,2,3\})<0$, $r_1 r_2>1$ and both $\delta_2$ and $\delta_3$ are sufficiently small positive numbers. 
Analogously, we obtain 
\begin{align}
&f_{\{1,2,3\}}(\ba(\{2,3\}) 
= -\frac{a_2(\{2,3\})}{c_1} \left( \delta_2-r_{3,2}^{\{2,3\}} \delta_3 \right)> 0, \\
&f_{\{1,3,4\}}(\ba(\{1,3,4\}) 
= -\frac{a_3(\{1,3,4\})}{c_1} \left( \delta_3-r_{4,3}^{\{1,3,4\}} \delta_4 \right)> 0, \\
&f_{\{1,3,4\}}(\ba(\{1,4\}) 
= -\frac{a_4(\{1,4\})}{c_1} \delta_4 > 0, 
\end{align}
where we use the assumptions for $\delta_2$, $\delta_3$ and $\delta_4$ and the facts that all $a_3(\{2,3\})$, $a_2(\{1,3,4\})$ and $a_4(\{1,4\})$ are negative. 
Furthermore, since $c_0$ is sufficiently large positive number and both $a_4(N)$ and $a_2(N)$ are negative, we have 
\begin{align}
&f_{\{1,2,3\}}(\ba(N)) 
= \frac{1}{c_1} \left( a_1(N)+\frac{1}{c_2} a_2(N)+\frac{1}{c_3} a_3(N)+\left(\frac{1}{c_4}-c_0\right) a_4(N) \right)>0, \\
&f_{\{1,3,4\}}(\ba(N)) 
= \frac{1}{c_1} \left( a_1(N)+\left(\frac{1}{c_2}-c_0\right) a_2(N)+\frac{1}{c_3} a_3(N)+\frac{1}{c_4} a_4(N) \right)>0. \label{eq:f134aN}
\end{align}
Using $f_{\{1,2,3\}}$ and $f_{\{1,3,4\}}$, we define a test function $f$ as, for $\by=(\bx,\bj)\in\calS$, 
\[
f(\by)=f(\bx,\bj)=\max\!\left\{f_{\{1,2,3\}}(\bx),\,f_{\{1,3,4\}}(\bx)\right\}, 
\]
and a subset $\calA$ of $\calS$ as 
\[
\calA = \{ \by\in\calS: f(\by)>1 \}.
\]
The hyperplane $f_{\{1,2,3\}}(\bx)=1$ intersects the $x_1$-axis at $(c_1,0,0,0)$, the $x_2$-axis at $(0,c_1 c_2,0,0)$, the $x_3$-axis at $(0,0,c_1 c_3,0)$ and the $x_4$-axis at $(0,0,0,-\frac{c_1 c_4}{c_0 c_4-1})$; the hyperplane $f_{\{1,3,4\}}(\bx)=1$ intersects the $x_1$-axis at $(c_1,0,0,0)$, the $x_2$-axis at $(0,-\frac{c_1 c_2}{c_0 c_2-1},0,0)$, the $x_3$-axis at $(0,0,c_1 c_3,0)$ and the $x_4$-axis at $(0,0,0,c_1 c_4)$. 
Therefore, since $c_0$ and $c_1$ are assumed to be positive and sufficiently large, it can be seen from the definition of the partition $\{\calV_A: A\subset N\}$ that we have $\calA\subset\calS\setminus\calV_\emptyset$. 
As mentioned in Section \ref{sec:modelandresults}, we have assumed that, for some $\bj\in S_0$, $((0,0,0,0),\bj)\in\calS_0$; this implies that, for every $x_1\in\mathbb{Z}_+$ and for some $\bj\in S_0$, $((x_1,0,0,0),\bj)\in\calS_0$, where $\calS_0$ is the irreducible class of $\{\bY_n\}$. 
Thus, it is obvious that $\calA\cap\calS_0\ne\emptyset$ and $\calA^C\cap\calS_0\ne\emptyset$. 

%%%%%%%%%%%%%%%%%%%%
We derive several properties of the test function $f$. Define a function $\sigma$ as, for $\bx=(x_1,x_2,x_3,x_4)\in\mathbb{R}^4$, 
\[
\sigma(\bx) = 
\left\{ \begin{array}{ll}
\{1,2,3\}, & x_2\ge x_4, \cr
\{1,3,4\}, & x_4> x_2.
\end{array} \right.
\]
Since we have 
\begin{align*}
&f_{\{1,2,3\}}(\bx)-f_{\{1,3,4\}}(\bx) = \frac{c_0}{c_1}(x_2-x_4), 
\end{align*}
$f(\by)$ is represented as, for $\by=(\bx,\bj)\in\calS$, 
\[
f(\by)=f(\bx,\bj)=f_{\sigma(\bx)}(\bx).
\]
If $\by=(\bx,\bj)\in\calV_{\{1,2,3\}}$, where $\bx=(x_1,x_2,x_3,x_4)\in\mathbb{Z}_+^4$, then we see, from the definition of $\calV_{\{1,2,3\}}$, that $x_1,x_2,x_3\ge K_{\{1,2,3\}}>K_N$ and $x_4\le K_N$, and this leads us to $x_2>x_4$; If $\by=(\bx,\bj)\in\calV_{\{2,3\}}$, then we see that $x_1<K_{\{1,2,3\}}$, $x_2,x_3\ge K_{\{2,3\}}>K_{\{2,3,4\}}$ and $x_4\le K_{\{2,3,4\}}$, and this also leads us to $x_2>x_4$. 
Analogously, for $\by=(\bx,\bj)\in\calV_{\{1,3,4\}}$ and $\by=(\bx,\bj)\in\calV_{\{1,4\}}$, we have $x_4>x_2$. Hence we obtain 
\begin{align}
f(\by) = f((\bx,\bj)) = \left\{ \begin{array}{ll} 
f_{\{1,2,3\}}(\bx) & \mbox{if $\by\in\calV_{\{1,2,3\}}$ or $\by\in\calV_{\{2,3\}}$}, \cr
f_{\{1,3,4\}}(\bx) & \mbox{if $\by\in\calV_{\{1,3,4\}}$ or $\by\in\calV_{\{1,4\}}$}.
\end{array} \right.
\label{eq:fanfA}
\end{align}
Furthermore, $f$ is a kind of sublinear function, since it satisfies, for $\by=(\bx,\bj),\by'=(\bx',\bj')\in\calS$, 
\begin{align}
f(\by')-f(\by) 
&= f_{\sigma(\bx')}(\bx')-f_{\sigma(\bx)}(\bx) \cr
&\ge f_{\sigma(\bx)}(\bx')-f_{\sigma(\bx)}(\bx) \cr
&= f_{\sigma(\bx)}(\bx'-\bx). 
\end{align}
Hence, letting $\phi(\by)$ be defined as  
\[
\phi(\by) 
= E(f(\tilde{\bY}_1)-f(\tilde{\bY}_0)\,|\,\tilde{\bY}_0=\by),\ \by\in\calS, 
\]
we have, for $\by=(\bx,\bj)\in\calS$,  
\begin{align*}
\phi(\by) 
&\ge E\Big( f_{\sigma(\bx)}\big(\tilde{\bX}_1-\tilde{\bX}_0\big)\,\big|\,\tilde{\bY}_0=\by \Big) 
= E\bigg(  f_{\sigma(\bx)}\Big(\sum_{n=1}^{\tau(\bx)} (\bX_n-\bX_{n-1}\Big)\,\Big|\,\bY_0=\by \bigg). 
\end{align*}
Therefore, considering inequalities (\ref{eq:f123a123}) to (\ref{eq:f134aN}) and expression (\ref{eq:fanfA}) and appropriately setting parameters $K_A$ and $T_A$ for $A\subset N,\,A\ne\emptyset$, we can see, in a way similar to that used for the proof of Theorem 5.2 of Ozawa \cite{Ozawa12}, that there exists a positive number $\varepsilon$ such that $\phi(\by)\ge\varepsilon$ for any $\by\in\calA$. 
Since $\{\bX_n\}$ is skip free in all coordinates, it is obvious that if $|f(\by_1)-f(\by_0)|>b$ then we have $P(\bY_1=\by_1\,|\,\bY_0=\by_0)=0$, where
\[
b = \frac{2}{c_1}\left(1+\frac{1}{c_2}+\frac{1}{c_3}+\frac{1}{c_4}+c_0\right).
\]
Hence, by Theorem \ref{th:Markov_transient}, the semi-irreducible Markov chain $\{\bY_n\}$ is transient.
\end{proof}

%%%%%%%%%%%%%%%%%%%%%%%
%
% Section 5
%
%%%%%%%%%%%%%%%%%%%%%%%
%
\section{Examples} \label{sec:examples}

We consider other two examples of service process: one is a service process with a well-known preemptive-resume priority service and the other that with a $(1,K)$-limited service. Assuming the service times of customers in each class are subject to a PH-distribution, we obtain the stability region of the two-station network depicted in Fig.~\ref{fig:twostation} with one of those service disciplines.

\subsection{Preemptive-resume priority service}

Consider the two-station network in which class-4 customers (class-2 customers) have preemptive-resume priority over class-1 customers (resp.\ class-3 customers). 
For $i\in N$, we assume the service times of class-$i$ customers are subject to a PH-distribution with the same representation as that of the non-preemptive service model explained in Section \ref{sec:modelandresults}; the mean service time of class-$i$ customers is given by $1/\mu_i$. 
The phase sets $S^s_{10}$, $S^s_{11}$, $S^s_{20}$ and $S^s_{23}$ are the same as those of the non-preemptive service model. On the other hand, in order to keep the phase of an interrupted service for a low-priority customer, $S^s_{14}$ and $S^s_{22}$ are given as $S^s_{14}=\{ s^s_1+2, s^s_1+3, ..., s^s_1+s^s_1 s^s_4+1 \}$ and $S^s_{22}=\{ s^s_3+2, s^s_3+3, ..., s^s_3+s^s_3 s^s_2+1 \}$. 
The representation of the MSP in station~1 is given in block form as follows (we omit several zeros in describing matrices). 
\[
\bar{T}_1^{00} = \begin{pmatrix} 
0 & & & &\cr 
\bone & -I & & & \cr
\bone & & -I & & \cr
\vdots & & & \ddots & \cr
\bone & & & & -I  \end{pmatrix},\quad 
\bar{T}_1^{+0} = \begin{pmatrix} 
-1 & \bbeta_1 & & &\cr 
 & \bar{H}_1 & & & \cr
 & \bone \be_1^\top & -I & & \cr
 & \vdots & & \ddots & \cr
 & \bone \be_{s^s_1}^\top & & & -I  \end{pmatrix}, 
\]
\[
\bar{T}_1^{1^*0} = \begin{pmatrix} 
0 & & & &\cr 
\bh_1 & O & & & \cr
 & & O & & \cr
 & & & \ddots & \cr
 & & & & O  \end{pmatrix},\quad 
\bar{T}_1^{2^*0} = \begin{pmatrix} 
0 & & & &\cr 
 & \bh_1 \bbeta_1 & & & \cr
 & & O & & \cr
 & & & \ddots & \cr
 & & & & O  \end{pmatrix}, 
\]
\[
\bar{T}_1^{0+} = \begin{pmatrix} 
-1 & & [\bbeta_1]_1 \bbeta_4 & \cdots & [\bbeta_1]_{s^s_1} \bbeta_4 \cr 
 & -I & [\bbeta_1]_1 \bone \bbeta_4 & \cdots & [\bbeta_1]_{s^s_1} \bone \bbeta_4 \cr
 & & \bar{H}_4 & & \cr
 & & & \ddots & \cr
 & & & & \bar{H}_4  \end{pmatrix},\quad 
\bar{T}_1^{01^*} = \begin{pmatrix} 
0 & & & &\cr 
 & O & & & \cr
\bh_4 & & O & & \cr
\vdots & & & \ddots & \cr
\bh_4 & & & & O  \end{pmatrix}, 
\]
\[
\bar{T}_1^{02^*} = \begin{pmatrix} 
0 & & & &\cr 
 & O & & & \cr
 & & \bh_4 \bbeta_4 & & \cr
 & & & \ddots & \cr
 & & & & \bh_4 \bbeta_4  \end{pmatrix},\quad 
\bar{T}_1^{++} = \begin{pmatrix} 
-1 & & [\bbeta_1]_1 \bbeta_4 & \cdots & [\bbeta_1]_{s^s_1} \bbeta_4 \cr 
 & -I & [\bbeta_1]_1 \bone \bbeta_4 & \cdots & [\bbeta_1]_{s^s_1} \bone \bbeta_4 \cr
 & & \bar{H}_4 & & \cr
 & & & \ddots & \cr
 & & & & \bar{H}_4  \end{pmatrix},
\]
\[
\bar{T}_1^{1^*+} = \bar{T}_1^{2^*+} = \begin{pmatrix} 
0 & & & &\cr 
 & O & & & \cr
 & & O & & \cr
 & & & \ddots &  \cr
 & & & & O  \end{pmatrix},\quad 
\bar{T}_1^{+1^*} = \begin{pmatrix} 
0 & & & &\cr 
 & O & & & \cr
 & \bh_4 \be_1^\top & O & & \cr
 & \vdots & & \ddots &  \cr
 & \bh_4 \be_{s^s_1}^\top & & & O  \end{pmatrix},
\]
\[
\bar{T}_1^{+2^*} = \begin{pmatrix} 
0 & & & &\cr 
 & O & & & \cr
 & & \bh_4 \bbeta_4 & & \cr
 & & & \ddots & \cr
 & & & & \bh_4 \bbeta_4 \end{pmatrix},\quad 
U_1^{0^*0} = \begin{pmatrix} 
0 & \bbeta_1 & & &\cr 
 & \bone \bbeta_1 & & & \cr
 & \bone \bbeta_1 & O & & \cr
 & \vdots & & \ddots & \cr
 & \bone \bbeta_1 & & & O \end{pmatrix},\quad 
\]
\[
U_1^{00^*} = \begin{pmatrix} 
0 & & [\bbeta_1]_1 \bbeta_4 & \cdots & [\bbeta_1]_{s^s_1} \bbeta_4 \cr 
 & O & [\bbeta_1]_1 \bone \bbeta_4 & \cdots & [\bbeta_1]_{s^s_1} \bone \bbeta_4 \cr
 & & [\bbeta_1]_1 \bone \bbeta_4 & \cdots & [\bbeta_1]_{s^s_1} \bone \bbeta_4 \cr
 & & \vdots & & \vdots \cr
 & & [\bbeta_1]_1 \bone \bbeta_4 & \cdots & [\bbeta_1]_{s^s_1} \bone \bbeta_4 \end{pmatrix},\quad 
U_1^{+0^*} = \begin{pmatrix} 
0 & & [\bbeta_1]_1 \bbeta_4 & \cdots & [\bbeta_1]_{s^s_1} \bbeta_4 \cr 
 & O & \be_1 \bbeta_4 & \cdots & \be_{s^s_1} \bbeta_4 \cr
 & & I & \cr
 & & & \ddots & \cr
 & & & & I \end{pmatrix}, 
\]
where $\be_i$ is a column vector whose $i$-th element is one and whose other elements are zero, i.e., the $i$-th unit vector; the dimension of $\be_i$ is determined in context. 
$U_1^{+^*0}$, $U_1^{0^*+}$, $U_1^{0+^*}$, $U_1^{+^*+}$ and $U_1^{++^*}$ become identity matrices. The representation of the MSP in station~2 are analogously given. 
In the Markov chain $\{\bY(t)\}$ arising from the two-station network with a preemptive-resume priority service, the state $((0,0,0,0),(1,1,1,1))$ is accessible from all other state; hence the Markov chain $\{\bY(t)\}$ is semi-irreducible. 

Assuming the nominal condition, we consider the case where $\mu_1>\mu_2$ and $\mu_3>\mu_4$; in this case, the nominal condition is not sufficient for the model to be stable. 
For $A\in\calN_p=\{N,\{1,2,3\},\{1,3,4\},\{1,4\},\{2,3\}\}$ and for $i\in A$, $\Delta q_i^A$ is given in a manner similar to that used in deriving those for the two-station network with a non-preemptive priority service, as follows.
\begin{align*}
& \Delta q_1^N=\lambda_1>0,\quad \Delta q_2^N=-\mu_2<0,\quad \Delta q_3^N=\lambda_3+p \mu_2>0,\quad \Delta q_4^N=-\mu_4<0, \cr
& \Delta q_1^{\{1,2,3\}}=\lambda_1-\mu_1<0,\quad \Delta q_2^{\{1,2,3\}}=\mu_1-\mu_2>0,\quad \Delta q_3^{\{1,2,3\}}=\lambda_3+p \mu_2>0, \cr
& \Delta q_1^{\{1,3,4\}}=\lambda_1>0,\quad \Delta q_3^{\{1,3,4\}}=\lambda_3-\mu_3<0,\quad \Delta q_4^{\{1,3,4\}}=\mu_3-\mu_4>0, \cr
& \Delta q_1^{\{1,4\}}=\lambda_1>0,\quad \Delta q_4^{\{1,4\}}=\lambda_3-\mu_4<0, \cr
& \Delta q_2^{\{2,3\}}=\lambda_1-\mu_2<0,\quad \Delta q_3^{\{2,3\}}=\lambda_3+p \mu_2>0. 
\end{align*}
These $\Delta q_i^A,\,A\in\calN_p,\,i\in A$, are the same as those of the two-station network with a non-preemptive priority service and we have 
\[
r_1 = \frac{\lambda_3+p \mu_2}{\mu_2-\lambda_1},\quad 
r_2 = \frac{\lambda_1}{\mu_4-\lambda_3}.  
\]
Since inequality $r_1 r_2<1$ is equivalent to $\rho_2+\rho_3<1$, we see that the two-station network with a preemptive-resume priority service is stable ($\{\bY(t)\}$ is positive recurrent) if $\rho_2+\rho_3<1$ and it is unstable ($\{\bY(t)\}$ is transient) if $\rho_2+\rho_3>1$; this result is coincident with the existing results (see, for example, Dai et al.\ \cite{Dai04} and references there).

%%%%%%%%%%%%%%%%%%%%%%%%%%%%%%%%%%%%%%%%%%%%%%%%%%%%%%%
\subsection{$(1,K)$-limited service}

We consider the two-station network in which customers in Q$_1$ (Q$_3$) are served according to a 1-limited service and those in Q$_4$ (resp.\ Q$_2$) according to a $K$-limited service, which means that the server in station 1 (resp.\ station 2) alternatively visits Q$_1$ and Q$_4$ (resp.\ Q$_3$ and Q$_2$), serves one customer upon a visit to Q$_1$ (resp.\ Q$_3$), and continuously serves customers upon a visit to Q$_4$ (resp.\ Q$_2$) until it completes serving just $K$ customers or Q$_4$ (resp.\ Q$_2$) becomes empty. 
We assume that switchover times of the servers are negligible. We call this service discipline a $(1,K)$-limited service. 
The $(1,K)$-limited service is equivalent to a non-preemptive service when the value of $K$ is unbounded (i.e., $K=\infty$). Using a $(1,K)$-limited service, we can control relative levels of priority between customers in different classes by varying the value of $K$.
For $i\in N$, we assume the service times of class-$i$ customers are subject to a PH-distribution with the same representation as that of the non-preemptive service model explained in Section \ref{sec:modelandresults}. 
The phase sets $S^s_{10}$, $S^s_{11}$, $S^s_{20}$ and $S^s_{23}$ are the same as those of the non-preemptive service model. On the other hand, in order to represent how many class-4 customers (class-2 customers) have continuously been served in Q$_4$ (resp.\ Q$_2$), $S^s_{14}$ and $S^s_{22}$ are given as $S^s_{14}=\{ s^s_1+2, s^s_1+3, ..., s^s_1+K s^s_4+1 \}$ and $S^s_{22}=\{ s^s_3+2, s^s_3+3, ..., s^s_3+K s^s_2+1 \}$. 
A representation of the MSP in station~1 is given in block form as follows (we omit several zeros in describing matrices). 
\[
\bar{T}_1^{00} = \begin{pmatrix} 
0 & & & &\cr 
\bone & -I & & & \cr
\bone & & -I & & \cr
\vdots & & & \ddots & \cr
\bone & & & & -I  \end{pmatrix},\quad 
\bar{T}_1^{+0} = \begin{pmatrix} 
-1 & \bbeta_1 & & &\cr 
 & \bar{H}_1 & & & \cr
 & \bone \bbeta_1 & -I & & \cr
 & \vdots & & \ddots & \cr
 & \bone \bbeta_1 & & & -I  \end{pmatrix}, 
\]
\[
\bar{T}_1^{1^*0} = \begin{pmatrix} 
0 & & & &\cr 
\bh_1 & O & & & \cr
 & & O & & \cr
 & & & \ddots & \cr
 & & & & O  \end{pmatrix},\quad 
\bar{T}_1^{2^*0} = \begin{pmatrix} 
0 & & & &\cr 
 & \bh_1 \bbeta_1 & & & \cr
 & & O & & \cr
 & & & \ddots & \cr
 & & & & O  \end{pmatrix}, 
\]
\[
\bar{T}_1^{0+} = \begin{pmatrix} 
-1 & & \bbeta_4 & &\cr 
 & -I & \bone \bbeta_4 & & \cr
 & & \bar{H}_4 & & \cr
 & & & \ddots & \cr
 & & & & \bar{H}_4  \end{pmatrix},\quad 
\bar{T}_1^{01^*} = \begin{pmatrix} 
0 & & & &\cr 
 & O & & & \cr
\bh_4 & & O & & \cr
\vdots & & & \ddots & \cr
\bh_4 & & & & O  \end{pmatrix}, 
\]
\[
\bar{T}_1^{02^*} = \begin{pmatrix} 
0 & & & &\cr 
 & O & & & \cr
 & & O & \bh_4 \bbeta_4 & \cr
 & & & \ddots & \ddots \cr
 & & \bh_4 \bbeta_4 & & O  \end{pmatrix},\quad 
\bar{T}_1^{++} = \begin{pmatrix} 
-1 & & \bbeta_4 & &\cr 
 & \bar{H}_1 & & & \cr
 & & \bar{H}_4 & & \cr
 & & & \ddots & \cr
 & & & & \bar{H}_4  \end{pmatrix},
\]
\[
\bar{T}_1^{1^*+} = \bar{T}_1^{2^*+} = \begin{pmatrix} 
0 & & & &\cr 
 & O & \bh_1 \bbeta_4 & & \cr
 & & O & & \cr
 & & & \ddots &  \cr
 & & & & O  \end{pmatrix},\quad 
\bar{T}_1^{+1^*} = \begin{pmatrix} 
0 & & & &\cr 
 & O & & & \cr
 & \bh_4 \bbeta_1 & O & & \cr
 & \vdots & & \ddots &  \cr
 & \bh_4 \bbeta_1 & & & O  \end{pmatrix},
\]
\[
\bar{T}_1^{+2^*} = \begin{pmatrix} 
0 & & & &\cr 
 & O & & & \cr
 & & O & \bh_4 \bbeta_4 & \cr
 & & & \ddots & \ddots  \cr
 & \bh_4 \bbeta_1 & & & O  \end{pmatrix},\quad 
U_1^{0^*0} = \begin{pmatrix} 
0 & \bbeta_1 & & &\cr 
 & \bone \bbeta_1 & & & \cr
 & \bone \bbeta_1 & O & & \cr
 & \vdots & & \ddots & \cr
 & \bone \bbeta_1 & & & O \end{pmatrix},\quad 
\]
\[
U_1^{00^*} = \begin{pmatrix} 
0 & & \bbeta_4 & & \cr 
 & O & \bone \bbeta_4 & & \cr
 & & \bone \bbeta_4 & & \cr
 & & \vdots & \ddots & \cr
 & & \bone \bbeta_4 & & O \end{pmatrix}. 
\]
$U_1^{+^*0}$, $U_1^{+0^*}$, $U_1^{0^*+}$, $U_1^{0+^*}$, $U_1^{+^*+}$ and $U_1^{++^*}$ become identity matrices. The representation of the MSP in station~2 are analogously given. 
In the Markov chain $\{\bY(t)\}$ arising from the two-station network with a $(1,K)$-limited service, the state $((0,0,0,0),(1,1,1,1))$ is accessible from all other state; hence the Markov chain $\{\bY(t)\}$ is semi-irreducible. 

For simplicity, we assume that model parameters are symmetric, which means that $\lambda_1=\lambda_3$, $\mu_1=\mu_3$, $\mu_2=\mu_4$ and $p=0$, and examine how the stability region of the two-station network is affected by the value of parameter $K$. 
We also assume the nominal condition and condition $\mu_1>\mu_2$. 
Furthermore, we assume parameter $K$ satisfies the following condition (the reason why this condition is necessary will become clear below):
\begin{align}
K>K^*=\max\!\left\{1,\,\frac{1-\rho_1}{\rho_2},\,\frac{\rho_1}{1-\rho_2} \right\}, 
\label{eq:Kcond}
\end{align}
where $\rho_1=\lambda_1/\mu_1$ and $\rho_2=\lambda_1/\mu_2$. Under this condition, for $A\in\calN_p$ and for $i\in A$, $\Delta q_i^A$ is given as follows. 
\begin{itemize}
%%%%%
\item $\Delta q_i^N,\,i\in N$. 
Consider the case where all the queues are saturated by customers. In station 1, after some finite time, the server serves one customer in Q$_1$ and then serves $K$ customers in Q$_4$; after that, the server repeats it. Thus, the output rates of Q$_1$ and Q$_4$ are given by 
\[
\bar{\mu}_1=\frac{1}{1/\mu_1+K/\mu_2},\quad 
\bar{\mu}_4=\frac{K}{1/\mu_1+K/\mu_2}. 
\]
Since the model is symmetric, we have $\bar{\mu}_2=\bar{\mu}_4$ and $\bar{\mu}_3=\bar{\mu}_1$. 
By expression (\ref{eq:inputrate}), we, therefore, obtain 
\begin{align*}
&\Delta q_1^N = \Delta q_3^N = \lambda_1-\frac{1}{1/\mu_1+K/\mu_2} = \frac{\rho_1+K\rho_2-1}{1/\mu_1+K/\mu_2}>0,\cr 
&\Delta q_2^N = \Delta q_4^N = \frac{1}{1/\mu_1+K/\mu_2}-\frac{K}{1/\mu_1+K/\mu_2}= -\frac{K-1}{1/\mu_1+K/\mu_2}<0, 
\end{align*}
where we use inequalities $K>(1-\rho_1)/\rho_2$ and $K>1$, which come from condition (\ref{eq:Kcond}). 
%%%%%
\item $\Delta q_i^{\{1,2,3\}},\,i\in {\{1,2,3\}}$, and $\Delta q_i^{\{1,3,4\}},\,i\in {\{1,3,4\}}$.  
Consider the case where queues Q$_1$, Q$_2$ and Q$_3$ are saturated by customers. In station 2, after some finite time, the server serves $K$ customers in Q$_2$ and then serves one customer in Q$_3$; after that, the server repeats it. Thus, the output rates of Q$_2$ and Q$_3$ are given by 
\[
\bar{\mu}_2=\frac{K}{1/\mu_1+K/\mu_2},\quad 
\bar{\mu}_3=\frac{1}{1/\mu_1+K/\mu_2}. 
\]
In station 1, we automatically obtain $\Delta q_4^N=0$ and this implies $\bar{\mu}_4=\bar{\mu}_3$. Hence, we obtain 
\[
\bar{\mu}_1 = \left( 1-\bar{\mu}_3/\mu_2 \right) \mu_1 = \frac{1+(K-1)\mu_1/\mu_2}{1/\mu_1+K/\mu_2}. 
\]
As a result, we obtain 
\begin{align*}
&\Delta q_1^{\{1,2,3\}} = \lambda_1 - \frac{1+(K-1)\mu_1/\mu_2}{1/\mu_1+K/\mu_2} = \frac{(K-1)(\rho_2-\mu_1/\mu_2)+\rho_1+\rho_2-1}{1/\mu_1+K/\mu_2}<0,\cr
&\Delta q_2^{\{1,2,3\}} = \frac{1+(K-1)\mu_1/\mu_2}{1/\mu_1+K/\mu_2} - \frac{K}{1/\mu_1+K/\mu_2} = \frac{(K-1)(\mu_1/\mu_2-1)}{1/\mu_1+K/\mu_2}>0,\cr
&\Delta q_3^{\{1,2,3\}} = \lambda_1 - \frac{1}{1/\mu_1+K/\mu_2} = \frac{\rho_1+K\rho_2-1}{1/\mu_1+K/\mu_2} > 0, 
\end{align*}
where we use the nominal condition, $\rho_1+\rho_2<1$, condition $\mu_1>\mu_2$ and inequality $K>(1-\rho_1)/\rho_2$, which comes from condition (\ref{eq:Kcond}). 
Since the model is symmetric, we obtain 
\begin{align*}
&\Delta q_1^{\{1,3,4\}} = \Delta q_3^{\{1,2,3\}}>0,\quad 
\Delta q_3^{\{1,3,4\}} = \Delta q_1^{\{1,2,3\}}<0,\quad
\Delta q_4^{\{1,3,4\}} = \Delta q_2^{\{1,2,3\}}>0.
\end{align*}
%%%%%
\item $\Delta q_i^{\{1,4\}},\,i\in {\{1,4\}}$, and $\Delta q_i^{\{2,3\}},\,i\in {\{2,3\}}$.  
Consider the case where queues Q$_1$ and Q$_4$ are saturated by customers. In station 1, after some finite time, the server serves one customer in Q$_1$ and then serves $K$ customers in Q$_4$; after that, the server repeats it. Thus, the output rates of Q$_1$ and Q$_4$ are given by 
\[
\bar{\mu}_1=\frac{1}{1/\mu_1+K/\mu_2},\quad 
\bar{\mu}_4=\frac{K}{1/\mu_1+K/\mu_2}. 
\]
In station 2, we automatically obtain $\Delta q_3^{\{1,4\}}=0$ and this implies $\bar{\mu}_3=\lambda_1$. Hence, we have 
\begin{align*}
&\Delta q_1^{\{1,4\}} = \lambda_1 - \frac{1}{1/\mu_1+K/\mu_2} = \frac{\rho_1+K\rho_2-1}{1/\mu_1+K/\mu_2} > 0,\cr
&\Delta q_4^{\{1,4\}} = \lambda_1 - \frac{K}{1/\mu_1+K/\mu_2} = \frac{\rho_1-K(1-\rho_2)}{1/\mu_1+K/\mu_2} < 0,
\end{align*}
where we use inequalities $K>(1-\rho_1)/\rho_2$ and $K>\rho_1/(1-\rho_2)$, which come from condition (\ref{eq:Kcond}). 
Since the model is symmetric, we obtain 
\begin{align*}
&\Delta q_2^{\{2,3\}} = \Delta q_4^{\{1,4\}}<0,\quad 
\Delta q_3^{\{2,3\}} = \Delta q_1^{\{1,4\}}>0.
\end{align*}
\end{itemize}

From the argument above, we see that, for $A\in\calN_P$ and for $i\in A$, $\Delta q_i^A$ satisfies the assumption of Theorem \ref{th:main0}; $r_1$ and $r_2$ in Theorem \ref{th:main0} are given by 
\[
r_1 = r_2 = \frac{\rho_1+K\rho_2-1}{-\rho_1+K(1-\rho_2)}, 
\]
and inequality $r_1 r_2<1$ is equivalent to 
\[
\left(\frac{\rho_1+K\rho_2}{1+K}\right)<\frac{1}{2}. 
\]
Thus, by Theorem \ref{th:main0}, if $\rho_2<1/2$, then $\{\bY(t)\}$ is positive recurrent for all $K>K^*$; if $\rho_2>1/2$, then it is positive recurrent for $K$ such that $K^*<K<\frac{1-2 \rho_1}{2 \rho_2-1}$ and transient for $K$ such that $K>\max\!\left\{K^*,\,\frac{1-2\rho_1}{2\rho_2-1} \right\}$.
The condition $\rho_2<1/2$ corresponds to $\rho_2+\rho_4<1$ in an asymmetric parameter case. Hence, we see that even though a two-station network with a non-preemptive priority service is unstable (transient), the corresponding two-station network with a $(1,K)$-limited service may be stable (positive recurrent) for some value of $K$. 
For example, when the parameters are set as $\lambda_1=1$, $\mu_1=5$ and $\mu_2=9/5$, we have $\rho_1+\rho_2=34/45<1$ and $\rho_2=5/9>1/2$. In this case, $\{\bY(t)\}$ is positive recurrent if $2\le K\le 5$ and it is transient if $K\ge 6$.

%%%%%%%%%%%%%%%%%%%%%%%
%
% References
%
%%%%%%%%%%%%%%%%%%%%%%%
%

%%%%%%%%%%%%%%%%%%%%%%%
%
% Appendix
%
%%%%%%%%%%%%%%%%%%%%%%%
\appendix

%%%%%%%%%%%%%%%%%%%%%%%
%
% Appendix: Q-matrix
%
%%%%%%%%%%%%%%%%%%%%%%%
\section{The infinitesimal generator of $\{\bY(t)\}$} \label{sec:qmatrix}

The infinitesimal generator (Q-matrix) of the Markov chain $\{\bY(t)\}$, $Q=(Q(\bx,\by),\,\bx,\by\in\mathbb{Z}_+^4)$, has the following non-zero blocks, where we denote by $\otimes$ the Kronecker product operation and by $\oplus$ the Kronecker sum operation (see, Bellman \cite{Bellman97}); for example, we have 
\[
A_1\oplus A_2\oplus A_3\oplus A_4 = 
A_1 \otimes I \otimes I \otimes I + I \otimes A_2 \otimes I \otimes I
+ I \otimes I \otimes A_3 \otimes I + I \otimes I \otimes I \otimes A_4.
\] 

{\it In the case where an exogenous customer arrives at Q$_1$:}
\begin{align*}
&Q((0,x_2,x_3,0),(1,x_2,x_3,0)) = \bar{D}_1 \otimes I \otimes U_1^{0^*0} \otimes I,\ x_2,x_3\ge 0, \cr
&Q((0,x_2,x_3,x_4),(1,x_2,x_3,x_4)) = \bar{D}_1 \otimes I \otimes U_1^{0^*+} \otimes I,\ x_4\ge 1,\ x_2,x_3\ge 0, \cr
&Q((x_1,x_2,x_3,0),(x_1+1,x_2,x_3,0)) = \bar{D}_1 \otimes I \otimes U_1^{+^*0} \otimes I,\ x_1\ge 1,\ x_2,x_3\ge 0, \cr
&Q((x_1,x_2,x_3,x_4),(x_1+1,x_2,x_3,x_4)) = \bar{D}_1 \otimes I \otimes U_1^{+^*+} \otimes I,\ x_1,x_4\ge 1,\ x_2,x_3\ge 0,
\end{align*}

%%%%%%%%%%%%%
{\it In the case where an exogenous customer arrives at Q$_3$:}
\begin{align*}
&Q((x_1,0,0,x_4),(x_1,0,1,x_4)) = I \otimes \bar{D}_3 \otimes I \otimes U_2^{0^*0},\ x_1,x_4\ge 0, \cr
&Q((x_1,x_2,0,x_4),(x_1,x_2,1,x_4)) = I \otimes \bar{D}_3 \otimes I \otimes U_2^{0^*+},\ x_1,x_4\ge 0,\ x_2\ge 1, \cr
&Q((x_1,0,x_3,x_4),(x_1,0,x_3+1,x_4)) = I \otimes \bar{D}_3 \otimes I \otimes U_2^{+^*0},\ x_1,x_4\ge 0,\ x_3\ge 1, \cr
&Q((x_1,x_2,x_3,x_4),(x_1,x_2,x_3+1,x_4)) = I \otimes \bar{D}_3 \otimes I \otimes U_2^{+^*+},\ x_1,x_4\ge 0,\ x_2,x_3\ge 1. 
\end{align*}

%%%%%%%%%%%%%%
{\it In the case where the service of a customer in Q$_1$ is completed:}
\begin{align*}
&Q((1,0,0,0),(0,1,0,0)) = I \otimes I \otimes \bar{T}_1^{1^*0} \otimes U_2^{00^*}, \cr
&Q((1,0,x_3,0),(0,1,x_3,0)) = I \otimes I \otimes \bar{T}_1^{1^*0} \otimes U_2^{+0^*},\ x_3\ge 1, \cr
&Q((1,x_2,0,0),(0,x_2+1,0,0)) = I \otimes I \otimes \bar{T}_1^{1^*0} \otimes U_2^{0+^*},\ x_2\ge 1, \cr
&Q((1,x_2,x_3,0),(0,x_2+1,x_3,0)) = I \otimes I \otimes \bar{T}_1^{1^*0} \otimes U_2^{++^*},\ x_2,x_3\ge 1,
\end{align*}
\begin{align*}
&Q((1,0,0,x_4),(0,1,0,x_4)) = I \otimes I \otimes \bar{T}_1^{1^*+} \otimes U_2^{00^*}, x_4\ge 1, \cr
&Q((1,0,x_3,x_4),(0,1,x_3,x_4)) = I \otimes I \otimes \bar{T}_1^{1^*+} \otimes U_2^{+0^*},\ x_3,x_4\ge 1, \cr
&Q((1,x_2,0,x_4),(0,x_2+1,0,x_4)) = I \otimes I \otimes \bar{T}_1^{1^*+} \otimes U_2^{0+^*},\ x_2,x_4\ge 1, \cr
&Q((1,x_2,x_3,x_4),(0,x_2+1,x_3,x_4)) = I \otimes I \otimes \bar{T}_1^{1^*+} \otimes U_2^{++^*},\ x_2,x_3,x_4\ge 1,
\end{align*}
\begin{align*}
&Q((x_1,0,0,0),(x_1-1,1,0,0)) = I \otimes I \otimes \bar{T}_1^{2^*0} \otimes U_2^{00^*},\ x_1\ge 2, \cr
&Q((x_1,0,x_3,0),(x_1-1,1,x_3,0)) = I \otimes I \otimes \bar{T}_1^{2^*0} \otimes U_2^{+0^*},\ x_1\ge 2,\ x_3\ge 1, \cr
&Q((x_1,x_2,0,0),(x_1-1,x_2+1,0,0)) = I \otimes I \otimes \bar{T}_1^{2^*0} \otimes U_2^{0+^*},\ x_1\ge 2,\ x_2\ge 1, \cr
&Q((x_1,x_2,x_3,0),(x_1-1,x_2+1,x_3,0)) = I \otimes I \otimes \bar{T}_1^{2^*0} \otimes U_2^{++^*},\ x_1\ge 2,\ x_2,x_3\ge 1, 
\end{align*}
\begin{align*}
&Q((x_1,0,0,x_4),(x_1-1,1,0,x_4)) = I \otimes I \otimes \bar{T}_1^{2^*+} \otimes U_2^{00^*},\ x_1\ge 2,\ x_4\ge 1, \cr
&Q((x_1,0,x_3,x_4),(x_1-1,1,x_3,x_4)) = I \otimes I \otimes \bar{T}_1^{2^*+} \otimes U_2^{+0^*},\ x_1\ge 2,\ x_3,x_4\ge 1, \cr
&Q((x_1,x_2,0,x_4),(x_1-1,x_2+1,0,x_4)) = I \otimes I \otimes \bar{T}_1^{2^*+} \otimes U_2^{0+^*},\ x_1\ge 2,\ x_2,x_4\ge 1, \cr
&Q((x_1,x_2,x_3,x_4),(x_1-1,x_2+1,x_3,x_4)) = I \otimes I \otimes \bar{T}_1^{2^*+} \otimes U_2^{++^*},\ x_1\ge 2,\ x_2,x_3,x_4\ge 1.
\end{align*}

%%%%%%%%%%%%%%
{\it In the case where the service of a customer in Q$_2$ is completed:}
\begin{align*}
&Q((x_1,1,0,x_4),(x_1,0,0,x_4)) = I \otimes I \otimes I \otimes (1-p)\bar{T}_2^{01^*},\ x_1,x_4\ge 0, \cr
&Q((x_1,1,x_3,x_4),(x_1,0,x_3,x_4)) = I \otimes I \otimes I \otimes (1-p)\bar{T}_2^{+1^*},\ x_1,x_4\ge 0,\ x_3\ge 1, \cr
&Q((x_1,x_2,0,x_4),(x_1,x_2-1,0,x_4)) = I \otimes I \otimes I \otimes (1-p)\bar{T}_2^{02^*},\ x_1,x_4\ge 0,\ x_2\ge 2, \cr
&Q((x_1,x_2,x_3,x_4),(x_1,x_2-1,x_3,x_4)) = I \otimes I \otimes I \otimes (1-p)\bar{T}_2^{+2^*},\ x_1,x_4\ge 0,\ x_2\ge 2,\ x_3\ge 1, 
\end{align*}
\begin{align*}
&Q((x_1,1,0,x_4),(x_1,0,1,x_4)) = I \otimes I \otimes I \otimes p\,\bar{T}_2^{01^*} U_2^{0^*0},\ x_1,x_4\ge 0, \cr
&Q((x_1,1,x_3,x_4),(x_1,0,x_3+1,x_4)) = I \otimes I \otimes I \otimes p\,\bar{T}_2^{+1^*} U_2^{+^*0},\ x_1,x_4\ge 0,\ x_3\ge 1, \cr
&Q((x_1,x_2,0,x_4),(x_1,x_2-1,1,x_4)) = I \otimes I \otimes I \otimes p\,\bar{T}_2^{02^*} U_2^{0^*+},\ x_1,x_4\ge 0,\ x_2\ge 2, \cr
&Q((x_1,x_2,x_3,x_4),(x_1,x_2-1,x_3+1,x_4)) = I \otimes I \otimes I \otimes p\,\bar{T}_2^{+2^*} U_2^{+^*+},\ x_1,x_4\ge 0,\ x_2\ge 2,\ x_3\ge 1.
\end{align*}

%%%%%%%%%%%%%%
{\it In the case where the service of a customer in Q$_3$ is completed:}
\begin{align*}
&Q((0,0,1,0),(0,0,0,1)) = I \otimes I \otimes U_1^{00^*} \otimes \bar{T}_2^{1^*0}, \cr
&Q((0,0,1,x_4),(0,0,0,x_4+1)) = I \otimes I \otimes U_1^{0+^*} \otimes \bar{T}_2^{1^*0},\ x_4\ge 1, \cr
&Q((x_1,0,1,0),(x_1,0,0,1)) = I \otimes I \otimes U_1^{+0^*} \otimes \bar{T}_2^{1^*0},\ x_1\ge 1, \cr
&Q((x_1,0,1,x_4),(x_1,0,0,x_4+1)) = I \otimes I \otimes U_1^{++^*} \otimes \bar{T}_2^{1^*0},\ x_1,x_4\ge 1, 
\end{align*}
\begin{align*}
&Q((0,x_2,1,0),(0,x_2,0,1)) = I \otimes I \otimes U_1^{00^*} \otimes \bar{T}_2^{1^*+}, x_2\ge 1, \cr
&Q((0,x_2,1,x_4),((0,x_2,0,x_4+1)) = I \otimes I \otimes U_1^{0+^*} \otimes \bar{T}_2^{1^*+},\ x_2,x_4\ge 1, \cr
&Q((x_1,x_2,1,0),(x_1,x_2,0,1)) = I \otimes I \otimes U_1^{+0^*} \otimes \bar{T}_2^{1^*+},\ x_1,x_2\ge 1, \cr
&Q((x_1,x_2,1,x_4),(x_1,x_2,0,x_4+1)) = I \otimes I \otimes U_1^{++^*} \otimes \bar{T}_2^{1^*+},\ x_1,x_2,x_4\ge 1, 
\end{align*}
\begin{align*}
&Q((0,0,x_3,0),(0,0,x_3-1,1)) = I \otimes I \otimes U_1^{00^*} \otimes \bar{T}_2^{2^*0},\ x_3\ge 2, \cr
&Q((0,0,x_3,x_4),(0,0,x_3-1,x_4+1)) = I \otimes I \otimes U_1^{0+^*} \otimes \bar{T}_2^{2^*0},\ x_3\ge 2,\ x_4\ge 1, \cr
&Q((x_1,0,x_3,0),(x_1,0,x_3-1,1)) = I \otimes I \otimes U_1^{+0^*} \otimes \bar{T}_2^{2^*0},\ x_3\ge 2,\ x_1\ge 1, \cr
&Q((x_1,0,x_3,x_4),(x_1,0,x_3-1,x_4+1)) = I \otimes I \otimes U_1^{++^*} \otimes \bar{T}_2^{2^*0},\ x_3\ge 2,\ x_1,x_4\ge 1, 
\end{align*}
\begin{align*}
&Q((0,x_2,x_3,0),(0,x_2,x_3-1,1)) = I \otimes I \otimes U_1^{00^*} \otimes \bar{T}_2^{2^*+},\ x_3\ge 2,\ x_2\ge 1, \cr
&Q((0,x_2,x_3,x_4),(0,x_2,x_3-1,x_4+1)) = I \otimes I \otimes U_1^{0+^*} \otimes \bar{T}_2^{2^*+},\ x_3\ge 2,\ x_2,x_4\ge 1, \cr
&Q((x_1,x_2,x_3,0),(x_1,x_2,x_3-1,1)) = I \otimes I \otimes U_1^{+0^*} \otimes \bar{T}_2^{2^*+},\ x_3\ge 2,\ x_1,x_2\ge 1, \cr
&Q(((x_1,x_2,x_3,x_4),(x_1,x_2,x_3-1,x_4+1)) = I \otimes I \otimes U_1^{++^*} \otimes \bar{T}_2^{2^*+},\ x_3\ge 2,\ x_1,x_2,x_4\ge 1.
\end{align*}

%%%%%%%%%%%%%%
{\it In the case where the service of a customer in Q$_4$ is completed:}
\begin{align*}
&Q((0,x_2,x_3,1),(0,x_2,x_3,0)) = I \otimes I \otimes \bar{T}_1^{01^*} \otimes I,\ x_2,x_3\ge 0, \cr
&Q((x_1,x_2,x_3,1),(0,x_2,x_3,0)) = I \otimes I \otimes \bar{T}_1^{+1^*} \otimes I,\ x_1\ge 1,\ x_2,x_3\ge 0, \cr
&Q((0,x_2,x_3,x_4),(0,x_2,x_3,x_4-1)) = I \otimes I \otimes \bar{T}_1^{02^*} \otimes I,\ x_2,x_3\ge 0,\ x_4\ge 2, \cr
&Q((x_1,x_2,x_3,x_4),(x_1,x_2,x_3,x_4-1)) = I \otimes I \otimes \bar{T}_1^{+2^*} \otimes I,\ x_1\ge 1,\ x_2,x_3\ge 0,\ x_4\ge 2.
\end{align*}

%%%%%%%%%%%%%%
{\it Diagonal blocks:}
\begin{align*}
&Q((0,0,0,0),(0,0,0,0)) = \bar{C}_1 \oplus \bar{C}_3 \oplus \bar{T}_1^{00} \oplus \bar{T}_2^{00}, \cr
&Q((0,0,0,x_4),(0,0,0,x_4)) = \bar{C}_1 \oplus \bar{C}_3 \oplus \bar{T}_1^{0+} \oplus \bar{T}_2^{00},\ x_4\ge 1, \cr
&Q((0,0,x_3,0),(0,0,x_3,0)) = \bar{C}_1 \oplus \bar{C}_3 \oplus \bar{T}_1^{00} \oplus \bar{T}_2^{+0},\ x_3\ge 1, \cr
&Q((0,0,x_3,x_4),(0,0,x_3,x_4)) = \bar{C}_1 \oplus \bar{C}_3 \oplus \bar{T}_1^{0+} \oplus \bar{T}_2^{+0},\ x_3,x_4\ge 1, 
\end{align*}
\begin{align*}
&Q((0,x_2,0,0),(0,x_2,0,0)) = \bar{C}_1 \oplus \bar{C}_3 \oplus \bar{T}_1^{00} \oplus \bar{T}_2^{0+},\ x_2\ge 1, \cr
&Q((0,x_2,0,x_4),(0,x_2,0,x_4)) = \bar{C}_1 \oplus \bar{C}_3 \oplus \bar{T}_1^{0+} \oplus \bar{T}_2^{0+},\ x_2,x_4\ge 1, \cr
&Q((0,x_2,x_3,0),(0,x_2,x_3,0)) = \bar{C}_1 \oplus \bar{C}_3 \oplus \bar{T}_1^{00} \oplus \bar{T}_2^{++},\ x_2,x_3\ge 1, \cr
&Q((0,x_2,x_3,x_4),(0,x_2,x_3,x_4)) = \bar{C}_1 \oplus \bar{C}_3 \oplus \bar{T}_1^{0+} \oplus \bar{T}_2^{++},\ x_2,x_3,x_4\ge 1, 
\end{align*}
\begin{align*}
&Q((x_1,0,0,0),(x_1,0,0,0)) = \bar{C}_1 \oplus \bar{C}_3 \oplus \bar{T}_1^{+0} \oplus \bar{T}_2^{00},\ x_1\ge 1, \cr
&Q((x_1,0,0,x_4),(x_1,0,0,x_4)) = \bar{C}_1 \oplus \bar{C}_3 \oplus \bar{T}_1^{++} \oplus \bar{T}_2^{00},\ x_1,x_4\ge 1, \cr
&Q((x_1,0,x_3,0),((x_1,0,x_3,0)) = \bar{C}_1 \oplus \bar{C}_3 \oplus \bar{T}_1^{+0} \oplus \bar{T}_2^{+0},\ x_1,x_3\ge 1, \cr
&Q((x_1,0,x_3,x_4),(x_1,0,x_3,x_4)) = \bar{C}_1 \oplus \bar{C}_3 \oplus \bar{T}_1^{++} \oplus \bar{T}_2^{+0},\ x_1,x_3,x_4\ge 1, 
\end{align*}
\begin{align*}
&Q((x_1,x_2,0,0),(x_1,x_2,0,0)) = \bar{C}_1 \oplus \bar{C}_3 \oplus \bar{T}_1^{+0} \oplus \bar{T}_2^{0+},\ x_1,x_2\ge 1, \cr
&Q((x_1,x_2,0,x_4),(x_1,x_2,0,x_4)) = \bar{C}_1 \oplus \bar{C}_3 \oplus \bar{T}_1^{++} \oplus \bar{T}_2^{0+},\ x_1,x_2,x_4\ge 1, \cr
&Q((x_1,x_2,x_3,0),(x_1,x_2,x_3,0)) = \bar{C}_1 \oplus \bar{C}_3 \oplus \bar{T}_1^{+0} \oplus \bar{T}_2^{++},\ x_1,x_2,x_3\ge 1, \cr
&Q((x_1,x_2,x_3,x_4),(x_1,x_2,x_3,x_4)) = \bar{C}_1 \oplus \bar{C}_3 \oplus \bar{T}_1^{++} \oplus \bar{T}_2^{++},\ x_1,x_2,x_3,x_4\ge 1. 
\end{align*}

%%%%%%%%%%%%%%%%%%%%%%%
%
% Appendix: the mean output rates
%
%%%%%%%%%%%%%%%%%%%%%%%
\section{The mean output rates of $\{\bY^A(t)\}$} \label{sec:output}

In this paper, we consider $\{\bY^A(t)\}$ only for $A=N,\{1,2,3\},\{1,3,4\},\{1,4\},\{2,3\}$. Hence, we state the expressions of the mean output rates only for such $A$'s. 

\medskip
In the case of $A=N=\{1,2,3,4\}$:
\begin{align*}
&\bar{\mu}_1^N = \bpi^N (I\otimes I\otimes \bar{T}_1^{2^*+}\otimes I)\,\bone, \quad 
\bar{\mu}_2^N = \bpi^N (I\otimes I\otimes I\otimes \bar{T}_2^{+2^*})\,\bone, \cr
&\bar{\mu}_3^N = \bpi^N (I\otimes I\otimes I\otimes \bar{T}_2^{2^*+})\,\bone, \quad 
\bar{\mu}_4^N = \bpi^N (I\otimes I\otimes \bar{T}_1^{+2^*}\otimes I)\,\bone.
\end{align*}

In the case of $A=\{1,2,3\}$:
\begin{align*}
&\bar{\mu}_1^{\{1,2,3\}} = \bpi^{\{1,2,3\}}(0) (I\otimes I\otimes \bar{T}_1^{2^*0}\otimes I)\,\bone 
+ \sum_{x_4=1}^\infty \bpi^{\{1,2,3\}}(x_4) (I\otimes I\otimes \bar{T}_1^{2^*+}\otimes I)\,\bone, \cr
&\bar{\mu}_2^{\{1,2,3\}} = \sum_{x_4=0}^\infty \bpi^{\{1,2,3\}}(x_4) (I\otimes I\otimes I\otimes \bar{T}_2^{+2^*})\,\bone, \cr
&\bar{\mu}_3^{\{1,2,3\}} = \sum_{x_4=0}^\infty \bpi^{\{1,2,3\}}(x_4) (I\otimes I\otimes I\otimes \bar{T}_2^{2^*+})\,\bone, \cr
&\bar{\mu}_4^{\{1,2,3\}} = \bpi^{\{1,2,3\}}(1) (I\otimes I\otimes \bar{T}_1^{+1^*}\otimes I)\,\bone 
+ \sum_{x_4=2}^\infty \bpi^{\{1,2,3\}}(x_4) (I\otimes I\otimes \bar{T}_1^{+2^*}\otimes I)\,\bone.
\end{align*}

In the case of $A=\{1,3,4\}$:
\begin{align*}
&\bar{\mu}_1^{\{1,3,4\}} = \sum_{x_2=0}^\infty \bpi^{\{1,3,4\}}(x_2) (I\otimes I\otimes \bar{T}_1^{2^*+}\otimes I)\,\bone, \cr
&\bar{\mu}_2^{\{1,3,4\}} = \bpi^{\{1,3,4\}}(1) (I\otimes I\otimes I\otimes \bar{T}_2^{+1^*})\,\bone 
+ \sum_{x_2=2}^\infty \bpi^{\{1,3,4\}}(x_2) (I\otimes I\otimes I\otimes \bar{T}_2^{+2^*})\,\bone, \cr
&\bar{\mu}_3^{\{1,3,4\}} = \bpi^{\{1,3,4\}}(0) (I\otimes I\otimes I\otimes \bar{T}_2^{2^*0})\,\bone 
+ \sum_{x_2=1}^\infty \bpi^{\{1,3,4\}}(x_2) (I\otimes I\otimes I\otimes \bar{T}_2^{2^*+})\,\bone, \cr
&\bar{\mu}_4^{\{1,3,4\}} = \sum_{x_2=0}^\infty \bpi^{\{1,3,4\}}(x_2) (I\otimes I\otimes \bar{T}_1^{+2^*}\otimes I)\,\bone.
\end{align*}

In the case of $A=\{1,4\}$:
\begin{align*}
&\bar{\mu}_1^{\{1,4\}} = \sum_{x_2=0}^\infty \sum_{x_3=0}^\infty \bpi^{\{1,4\}}(x_2,x_3) (I\otimes I\otimes \bar{T}_1^{2^*+}\otimes I)\,\bone, \cr
&\bar{\mu}_2^{\{1,4\}} = \bpi^{\{1,4\}}(1,0) (I\otimes I\otimes I\otimes \bar{T}_2^{01^*})\,\bone 
+ \sum_{x_3=1}^\infty \bpi^{\{1,4\}}(1,x_3) (I\otimes I\otimes I\otimes \bar{T}_2^{+1^*})\,\bone \cr
&\qquad +\sum_{x_2=2}^\infty \bpi^{\{1,4\}}(x_2,0) (I\otimes I\otimes I\otimes \bar{T}_2^{02^*})\,\bone 
+ \sum_{x_2=2}^\infty \sum_{x_3=1}^\infty \bpi^{\{1,4\}}(x_2,x_3) (I\otimes I\otimes I\otimes \bar{T}_2^{+2^*})\,\bone, \cr
&\bar{\mu}_3^{\{1,4\}} = \bpi^{\{1,4\}}(0,1) (I\otimes I\otimes I\otimes \bar{T}_2^{1^*0})\,\bone 
+ \sum_{x_2=1}^\infty \bpi^{\{1,4\}}(x_2,1) (I\otimes I\otimes I\otimes \bar{T}_2^{1^*+})\,\bone \cr
&\qquad +\sum_{x_3=2}^\infty \bpi^{\{1,4\}}(0,x_3) (I\otimes I\otimes I\otimes \bar{T}_2^{2^*0})\,\bone 
+ \sum_{x_2=1}^\infty \sum_{x_3=2}^\infty \bpi^{\{1,4\}}(x_2,x_3) (I\otimes I\otimes I\otimes \bar{T}_2^{2^*+})\,\bone, \cr
&\bar{\mu}_4^{\{1,4\}} = \sum_{x_2=0}^\infty \sum_{x_3=0}^\infty \bpi^{\{1,4\}}(x_2,x_3) (I\otimes I\otimes \bar{T}_1^{+2^*}\otimes I)\,\bone.
\end{align*}

In the case of $A=\{2,3\}$:
\begin{align*}
&\bar{\mu}_1^{\{2,3\}} = \bpi^{\{2,3\}}(1,0) (I\otimes I\otimes \bar{T}_1^{1^*0}\otimes I)\,\bone 
+ \sum_{x_4=1}^\infty \bpi^{\{2,3\}}(1,x_4) (I\otimes I\otimes \bar{T}_1^{1^*+}\otimes I)\,\bone \cr
&\qquad +\sum_{x_1=2}^\infty \bpi^{\{2,3\}}(x_1,0) (I\otimes I\otimes \bar{T}_1^{2^*0}\otimes I)\,\bone 
+ \sum_{x_1=2}^\infty \sum_{x_4=1}^\infty \bpi^{\{2,3\}}(x_1,x_4) (I\otimes I\otimes \bar{T}_1^{2^*+}\otimes I)\,\bone, \cr
&\bar{\mu}_2^{\{2,3\}} = \sum_{x_1=0}^\infty \sum_{x_4=0}^\infty \bpi^{\{2,3\}}(x_1,x_4) (I\otimes I\otimes I\otimes \bar{T}_2^{+2^*})\,\bone, \cr
&\bar{\mu}_3^{\{2,3\}} = \sum_{x_1=0}^\infty \sum_{x_4=0}^\infty \bpi^{\{2,3\}}(x_1,x_4) (I\otimes I\otimes I\otimes \bar{T}_2^{2^*+})\,\bone, \cr
&\bar{\mu}_4^{\{2,3\}} = \bpi^{\{2,3\}}(0,1) (I\otimes I\otimes \bar{T}_1^{01^*}\otimes I)\,\bone 
+ \sum_{x_1=1}^\infty \bpi^{\{2,3\}}(x_1,1) (I\otimes I\otimes \bar{T}_1^{+1^*}\otimes I)\,\bone \cr
&\qquad +\sum_{x_4=2}^\infty \bpi^{\{2,3\}}(0,x_4) (I\otimes I\otimes \bar{T}_1^{02^*}\otimes I)\,\bone 
+ \sum_{x_1=1}^\infty \sum_{x_4=2}^\infty \bpi^{\{2,3\}}(x_1,x_4) (I\otimes I\otimes \bar{T}_1^{+2^*}\otimes I)\,\bone. 
\end{align*}

\end{document}